\documentclass[12pt]{article}
\usepackage{amsthm,amssymb,amsmath,txfonts,textcomp,mathrsfs,color}
\usepackage{mathrsfs}
\usepackage{amsmath}
\usepackage{amsfonts}
\usepackage{amssymb}
\usepackage{amsmath,amssymb,amsthm}
\usepackage{amsmath,amssymb,amsthm,amscd}
\usepackage[frame,cmtip,arrow,matrix,line,graph,curve]{xy}
\usepackage{graphpap, color}
\usepackage[mathscr]{eucal}
\usepackage{color}
\usepackage{verbatim}
\usepackage{hyperref}
\usepackage{cite}
\usepackage{slashbox,multirow}
\usepackage{float}
\newtheoremstyle{mythm}{3pt}{3pt}{}{}{\bfseries}{}{5mm}{}

\usepackage{float}
\catcode`!=11
\let\!iint\iint \def\iint{\displaystyle\!iint}
\let\!int\int   \def\int{\displaystyle\!int}
\let\!lim\lim   \def\lim{\displaystyle\!lim}
\let\!sum\sum   \def\sum{\displaystyle\!sum}
\let\!sup\sup   \def\sup{\displaystyle\!sup}
\let\!inf\inf   \def\inf{\displaystyle\!inf}
\let\!cap\cap   \def\cap{\displaystyle\!cap}
\let\!max\max   \def\max{\displaystyle\!max}
\catcode`!=12

\newtheorem{thm}{Theorem}[section]
\newtheorem{lem}[thm]{Lemma}

\newtheorem{prop}[thm]{Proposition}
\newtheorem{corr}[thm]{Corollary}
\newtheorem{con1}[thm]{Conjecture}

\newtheorem{con}[thm]{Conjecture}
\theoremstyle{remark}\newtheorem{rem}{Remark}[section]
\newtheorem*{ack}{Acknowledgment}

\allowdisplaybreaks

\numberwithin{equation}{section}
\begin{document}

\title{The Gap of the Consecutive Eigenvalues of the  \\ Drifting Laplacian on  Metric Measure Spaces}
\author{Lingzhong Zeng
\\ \small   College of Mathematics and Informational Science,
Jiangxi Normal University,
\\
\small Nanchang 330022, China, E-Mail: lingzhongzeng@yeah.net
}
\date{}

\maketitle

\begin{abstract}\noindent In this paper, we investigate eigenvalues of the Dirichlet problem and the closed eigenvalue problem of drifting Laplacian
on  the complete metric measure spaces and establish the corresponding general formulas. By using those general formulas, we give
some upper bounds of consecutive gap of the eigenvalues of the eigenvalue problems, which is sharp in the sense of the order of the eigenvalues. As some interesting applications, we study the eigenvalue of
drifting Laplacian on Ricci solitons, self-shrinkers and product Riemannian manifolds. We give the explicit upper bounds of the gap of the consecutive eigenvalues of the drifting Laplacian.
Since eigenvalues is  invariant in the sense of isometry, by the classifications of Ricci solitons and self-shrinkers,
we give the explicit upper bounds for the consecutive eigenvalues of the
drifting Laplacian on a large class metric measure spaces. In addition, we also consider the case of product Riemannian manifolds with certain curvature conditions and some upper bounds are
obtained.
Basing on the case of Laplace operator, we also present a conjecture as follows: all  of the eigenvalues  of
the Dirichlet problem of drifting Laplacian on metric measure spaces satisfy:
$$\lambda_{k+1}-\lambda_{k}\leq(\lambda_{2}-\lambda_{1})k^{\frac{1}{n}}.$$We note the conjecture is true in some special cases.

\vskip3mm
\noindent 2000 MSC 35P15, 53C40.
\end{abstract}

\section{Introduction}
Let $M^{n}$ be an $n$-dimensional complete Riemannian manifold with smooth metric $g$,
 and $\Omega$ is a bounded domain with piecewise smooth boundary $\partial\Omega$. We consider the following Dirichlet problem:

\begin{equation}
\left\{ \begin{aligned} \label{Eigen-Prob-Lapl} \Delta u=-\lambda u,\ \ &{\rm
in}\ \ \ \ \Omega,
         \\
u=0,\ \ &{\rm on}\ \ \partial\Omega,
                          \end{aligned} \right.
                          \end{equation}where $$\Delta=\frac{1}{\sqrt{det(g)}}\sum_{i,j=1}\partial_{i}g^{ji}\sqrt{det(g)}\partial_{j}.$$
If $M^{n}$ is an $n$-dimensional Euclidean space $\mathbb{R}^{n}$, L. E. Payne,
G. P\'{o}lya and H. F. Weinberger \cite{PPW1} and \cite{PPW2} investigated the eigenvalue inequalities of the Dirichlet problem \eqref{Eigen-Prob-Lapl}.
They established the following universal inequality:
\begin{equation}\label{ppw-ineq}\lambda_{k+1}-\lambda_{k}\leq\frac{4}{nk}\sum^{k}_{i=1}\lambda_{i}.\end{equation}
In various  backgrounds, many mathematicians extended Payne,
P\'{o}lya and Weinberger's universal inequality. However, among a large amount of literatures, there are two main contributions due to G. N. Hile and M. H. Protter \cite{HP} and
H.-C. Yang \cite{Y}. In 1980, G. N. Hile and M. H.  Protter proved the following universal inequality of eigenvalues:

\begin{equation}\label{hp-ineq}\sum^{k}_{i=1}\frac{\lambda_{i}}{\lambda_{k+1}-\lambda_{i}}\geq\frac{nk}{4}.\end{equation}
After a direct calculation, one can show that inequality \eqref{hp-ineq} implies inequality \eqref{ppw-ineq}.
In 1991, H.-C. Yang  proved a very sharp universal inequality in his famous paper \cite{Y} (cf. \cite{CY2}):

\begin{equation}\label{y1-ineq}\sum^{k}_{i=1}(\lambda_{k+1}-\lambda_{i})^{2}\leq\frac{4}{n}\sum^{k}_{i=1}(\lambda_{k+1}-\lambda_{i})\lambda_{i},\end{equation}
which is called H.-C. Yang's first inequality  by M. S. Ashbaugh (cf. \cite{A1}, \cite{A2}).
From \eqref{y1-ineq}, one can infer that
\begin{equation}\label{y2-ineq}\lambda_{k+1}\leq\frac{1}{k}(1+\frac{4}{n})\sum^{k}_{i=1}\lambda_{i},\end{equation}
which is called  H.-C. Yang's   second inequality (cf. \cite{A1}, \cite{A2}). In 2007, Q.-M. Cheng and H.-C. Yang established a celebrated recursion formula \cite{CY2}.
By utilizing this recursion formula,  they gave an explicit upper bound:

\begin{equation}
\begin{aligned}
\label{Cheng-Yang-ineq}\lambda_{k+1}
\leq C_{0}(n,k)k^{\frac{2}{n}}\lambda_{1},
\end{aligned}
\end{equation}
where the constant $C_{0}(n,k)\leq1+\frac{4}{n}$ only depend on $n$ and $k$ (see Q.-M. Cheng and H.-C. Yang's paper \cite{CY2}).
Let $\Omega$ be a bounded domain on an $n$-dimensional Euclidean space $\mathbb{R}^{n}$ or hyperbolic space. For this assumption, in 2016, D. Chen, T. Zheng and H.-C. Yang \cite{CZY}
obtained an upper for the gap of consecutive eigenvalues of the eigenvalue problem \eqref{Eigen-Prob-Lapl} as follows:

\begin{equation}\label{czy-1}\begin{aligned}\lambda_{k+1}-\lambda_{k}\leq
C_{n,\Omega}k^{\frac{1}{n}},\end{aligned}
\end{equation}
where $$C_{n,\Omega}=4\lambda_{1}\sqrt{\frac{C_{0}(n)}{n}},$$ and the constant $C_{0}(n)$ is the same as the one in \eqref{Cheng-Yang-ineq}.
It is well known that the order of the upper
bound of the gap of the consecutive eigenvalues of $\mathbb{S}^{n}$ with
standard metric is $k^{\frac{1}{n}}$. Therefore, for general Riemannian manifolds, D. Chen, T. Zheng and H.-C. Yang  proposed the following conjecture in the same paper \cite{CZY}:

\begin{con1}
\label{conj1}Let $(M^{n},g,f)$ be a complete smooth measure space and $\lambda_{i}$ be the $i$-th $(i=1,2,\cdots,k)$
eigenvalue of the eigenvalue problem \eqref{Eigen-Prob-Lapl}. Then we have

\begin{equation*}\begin{aligned}\lambda_{k+1}-\lambda_{k}\leq
C_{n,\Omega}k^{\frac{1}{n}},\end{aligned}
\end{equation*}
where $$C_{n,\Omega}=4(\lambda_{1}+c_{1})\sqrt{\frac{C_{0}(n)}{n}}.$$

\end{con1}
Furthermore, by constructing a new trial function, the author recently made an affirmative answer to  this conjecture in \cite{Z2}.

\vskip 3mm

Let $M^{n}$ be an $n$-dimensional closed Riemannian manifold. We consider the closed eigenvalue problem of Laplacian:

\begin{equation}
 \label{Eigen-Prob-closed}\Delta u=-\overline{\lambda} u,\ \ {\rm
in}\ \ \ \ M^{n}.\end{equation}
It is well known that the eigenvalues of the closed eigenvalue problem  \eqref{Eigen-Prob-closed}  is
discrete and satisfies the following:

\begin{equation*}
0=\overline{\lambda}_{0}<\overline{\lambda}_{1}\leq\overline{\lambda}_{2}\leq\overline{\lambda}_{3}\leq\cdots\leq\overline{\lambda}_{k}\leq\cdots\rightarrow+\infty,
\end{equation*}
where $\overline{\lambda}_{k}$ is the $k$-th eigenvalue of the closed eigenvalue problem \eqref{Eigen-Prob-closed} and each eigenvalue is repeated according to its multiplicity.
We assume that $M^{n}$ is an $n$-dimensional compact homogeneous Riemannian manifold. In 1980, P. Li\cite{L} investigated the closed eigenvalue problem \eqref{Eigen-Prob-closed} and  proved
the following universal inequality:
\begin{equation*}\begin{aligned}\overline{\lambda}_{k+1}-\overline{\lambda}_{k}\leq
\frac{2}{k+1}\left(\sqrt{\left(\sum^{k}_{i=1}\overline{\lambda}_{i}\right)^{2}+(k+1)\sum^{k}_{i=1}\overline{\lambda}_{i}\overline{\lambda}_{1}}
+\sum^{k}_{i=1}\overline{\lambda}_{i}\right)+\overline{\lambda}_{1}.\end{aligned}
\end{equation*}
 If $M^{n}$ is
an $n$-dimensional compact minimal submanifold in a unit sphere $\mathbb{S}^{N}(1)$, then, in 1980,  P. C.Yang and S. T. Yau \cite{YY} proved the eigenvalues of the closed eigenvalue problem satisfy the following eigenvalue inequality:
\begin{equation*}\begin{aligned}\overline{\lambda}_{k+1}-\overline{\lambda}_{k}\leq n+
\frac{2}{n(k+1)}\left(\sqrt{\left(\sum^{k}_{i=1}\overline{\lambda}_{i}\right)^{2}+n^{2}(k+1)
\sum^{k}_{i=1}\overline{\lambda}_{i}\overline{\lambda}_{1}}+\sum^{k}_{i=1}\overline{\lambda}_{i}\right).\end{aligned}\end{equation*}
Furthermore, E. M. Harrel II and P. L. Michel and J. Stubbe (see (\cite{HM} 1994) and (\cite{H3}1997 )) obtained an abstract inequality of algebraic version.
By applying the algebraic inequality, they proved that,
if $M^{n}$ is an $n$-dimensional compact minimal submanifold in a unit
sphere $\mathbb{S}^{N}(1)$, one  has the following eigenvalue inequality:

\begin{equation}\begin{aligned}\label{hm-ineq-1}\overline{\lambda}_{k+1}-\overline{\lambda}_{k}\leq n+
\frac{4}{n(k+1)}\sum^{k}_{i=1}\overline{\lambda}_{i},\end{aligned}
\end{equation}
and if $M^{n}$ is an $n$-dimensional compact homogeneous Riemannian manifold, then we have

\begin{equation}\begin{aligned}\label{hm-ineq-2}\overline{\lambda}_{k+1}-\overline{\lambda}_{k}\leq
\frac{4}{k+1}\sum^{k}_{i=1}\overline{\lambda}_{i}+\overline{\lambda}_{1},\end{aligned}
\end{equation} One can easily to see that the above inequalities \eqref{hm-ineq-1} and \eqref{hm-ineq-2} made significant improvement to earlier
estimates of differences of consecutive eigenvalues of Laplacian introduced by
P. C. Yang and S. T. Yau  \cite{YY}, P.-F. Leung  \cite{Le}, P. Li \cite{L} and E. M. Harrel II \cite{H1}.
Q.-M. Cheng and H.-C. Yang also considered the same eigenvalue problem and proved that, when $M^{n}$
is an $n$-dimensional compact homogeneous Riemannian manifold  without  boundary, then the eigenvalues of the
close eigenvalue problem \eqref{Eigen-Prob-closed} satisfy

\begin{equation*}\begin{aligned}\overline{\lambda}_{k+1}-\overline{\lambda}_{k}\leq\left[\left(
\frac{4}{k+1}\sum^{k}_{i=1}\overline{\lambda}_{i}+\overline{\lambda}_{1}\right)^{2}-\frac{20}{k+1}\sum^{k}_{i=0}\left(\overline{\lambda}_{i}
-\frac{1}{k+1}\sum^{k}_{j=1}\overline{\lambda}_{j}\right)^{2}\right]^{\frac{1}{2}};\end{aligned}
\end{equation*}
and when $M^{n}$ is an $n$-dimensional compact minimal submanifold without boundary in a unit sphere $\mathbb{S}^{N}(1)$, then the eigenvalues of the
close eigenvalue problem \eqref{Eigen-Prob-closed} satisfy

\begin{equation*}\begin{aligned}\overline{\lambda}_{k+1}-\overline{\lambda}_{k}\leq2\left[\left(
\frac{2}{n}\frac{1}{k}\sum^{k}_{i=0}\overline{\lambda}_{i}+\frac{n}{2}\right)^{2}-\left(1+\frac{4}{n}\right)\frac{1}{k+1}\sum^{k}_{j=0}\left(\overline{\lambda}_{j}
-\frac{1}{k}\sum^{k}_{i=0}\overline{\lambda}_{i}\right)^{2}\right]^{\frac{1}{2}}.\end{aligned}
\end{equation*} In \cite{Z2},
the author  studied the closed eigenvalue problem \eqref{Eigen-Prob-closed} of
Laplacian and obtained a similar optimal upper bound. As a further interest, the author also investigated the eigenvalues of the Laplacian
on compact homogeneous Riemannian manifolds without boundary in \cite{Z2}.

We suppose that $f$ is a smooth function on $M^{n}$. The
triple $(M^{n}, g, e^{-f} dv)$ is called  a metric measure space
with weighted volume density $e^{-f}dv$. Furthermore, we say that the
triple $(M^{n}, g, e^{-f} dv)$ is an $n$-dimensional complete metric measure space if $M^{n}$ is a complete Riemannian manifold with dimension $n$, while the
triple $(M^{n}, g, e^{-f} dv)$ is an $n$-dimensional closed metric measure space if $M^{n}$ is a closed Riemannian manifold with dimension $n$.  The metric measure spaces also arise in smooth
collapsed Gromov-Hausdorff limits. So-called Bakry-\'{E}mery Ricci
tensor ${\rm Ric}^{f}$ corresponding to weighted metric measure
spaces is a very important curvature quantity, which is defined by
\begin{equation}
\label{1.1}{\rm Ric}^{f}:={\rm
Ric}+{\rm Hess}f,
\end{equation}
where ${\rm Ric}$  and ${\rm Hess}f$ denote  Ricci tensor of $M^n$
and Hessian of $f$, respectively  (see
\cite{BE, LW}).
When $f$ is a constant, we have
\begin{equation}
\label{1.2}{\rm Ric}^{f}= {\rm Ric}.
\end{equation}
Therefore, the Bakry-\'{E}mery Ricci tensor  is naturally viewed as
an extension of the Ricci tensor. Recently,  a great deal of
significant  results under assumption on the Bakry-\'{E}mery Ricci
tensor have  been obtained. For instances,   A. Lichnerowicz
\cite{Lic1,Lic2} has extended the classical Cheeger-Gromoll
splitting theorem   to the metric measure spaces with ${\rm Ric}^f
\geq 0$ and $f$ is bounded,
 G. F. Wei and W. Wylie in \cite{WW}
have proved the weighted volume comparison theorems; O. Munteanu and J. Wang \cite{MW1, MW2} have established
 gradient estimates for positive weighted harmonic functions.
The metric measure space has studied by many geometric analysis  (cf:
\cite{AN,B,CH2,CaZ,CP,MW1,MW2,W,W2}) during the last twenty years. Next, we give definition of the drifting Laplacian associated
with the metric measure space:
\begin{equation*}
\Delta_{f}u:=\Delta u-\langle\nabla f,\nabla u\rangle=e^{f}{\rm
div}\left(e^{-f}\nabla u\right).
\end{equation*}
It is not difficult to see that drifting Laplacian is a self-adjoint operator with
respect to the weighted volume measure $e^{-f}dv$, i.e.,\begin{equation}
\begin{aligned}
\label{self-adjiont} -\int_{M^{n}}\langle\nabla u,\nabla w\rangle e^{-f}dv
=\int_{M^{n}}u(\Delta_{f}w)e^{-f}dv=\int_{M^{n}}w(\Delta_{f}u)e^{-f}dv,
\end{aligned}
\end{equation}
and it is an important elliptic operator which is
widely used in the probability theory and geometrical analysis. In particular, many mathematicians pay more and more attention to the research of eigenvalue of
the drifting Laplacian in recent years. For this recent developments, we refer to \cite{AN,CL,FLL,MD,LW,MW1,MW2, SZ,WW,W1}  and the references therein.
On one hand, L. Ma and S.-H. Du \cite{MD} and H. Li and
Y. Wei \cite{LW} have studied  the  Reilly formula of the
Witten-Laplacian version  to obtain a lower bound of the first
eigenvalue for the Witten-Laplacian on the $f$-minimal hypersurface.
Furthermore,  they have given  a Lichnerowicz type lower bound for
the first eigenvalue of the Witten-Laplacian on compact manifolds
with positive Bakry-\'{E}mery Ricci curvature. In 2013, A. Futaki and Y. Sano
\cite{FuS}  have  studied the lower bound of the first eigenvalue of
the Witten-Laplacian on compact manifolds $M^n$ if the
Bakry-\'{E}mery Ricci curvature bounded from below by $(n-1)K$ and
obtained the following:
\begin{equation}\label{fs-in}
\lambda_{1}\geq\frac{\pi^{2}}{d^{2}}+\frac{31K}{100},
\end{equation}
and A. Futaki, H. Li and X.-D. Li \cite{FLL} (cf. \cite{AN})  have also improved the above result to
\begin{equation*}
\lambda_{1}\geq\sup_{s\in(0,1)}\Big{\{}
4s(1-s)\frac{\pi^{2}}{d^{2}}+sK\Big{\}},
\end{equation*}
where $d$ is
the diameter of $(M^{n},g)$. As an application,  an upper  bound
of the diameter  of $(M^{n},g)$ has been obtained. In addition, under the assumption $Ric_{f} \geq -(n-1)k$ for some $k\geq 0$, N. Charalambous, Z. Lu  and J. Rowlett obtained \cite{CLR}:

\begin{equation}\label{first-eigen}\lambda_{1}\geq
\frac{\pi^{2}}{d^{2}}
\exp(-c_{n}\sqrt{kd^{2}}),\end{equation}
where $d$ is the diameter of $M$ with respect to $g$, and $c_{n}$ is a constant depending only
on $n$.
In \cite{CLR},   N. Charalambous, Z. Lu  and J. Rowlett proved the Bakry-\'{E}mery maximum principle. Applying this result, they proved the eigenvalue inequality \eqref{fs-in}
given by A. Futaki and Y. Sano \cite{FS}. We note that the corresponding
Riemannian case is proved by J. Ling \cite{Ling}. On the other hand,  upper bounds for
the first eigenvalue of the drifting Laplacian on complete Riemannian
manifolds  have been studied in   \cite{MW1,MW2, SZ,W1}. In
particular, J. Wu in \cite{W1} (also see \cite{W2}) established an upper bounds for the first
eigenvalue of the drifting  Laplacian on compact gradient Ricci soliton
if $f$ is bounded. Assume that $(M^{n},g,f)$ is a compact metric measure space without boundary, and  $\epsilon> 0$. If
$$Ric_{f}-\epsilon\nabla f\otimes\nabla f\geq-(n-1)K, for  \ \ K\geq0,$$
then we have (see \cite{CLR}) the following estimate:
$$\lambda_{k}\leq C(n,\epsilon)(K+ k^{2}/d^{2}), \forall k\in \mathbb{N},$$
where $d$ is the diameter of $M$ and $C(n,\epsilon)$ is a constant depending on $n$ and $\epsilon$.
Furthermore, we assume that $K=0$, then, by using make of \eqref{first-eigen}, we have

\begin{equation}\label{FS-eigen}\lambda_{k}\leq C(n,\epsilon)\lambda_{1}.\end{equation}
If $(M,g,f)$ is a compact Bakry-\'{E}mery manifold with non-negative Bakry-\'{E}mery Ricci curvature, then, in 2013,
K. Funano and T. Shioya proved \cite{FS} the following stronger and somewhat
surprising inequality:
$$\lambda_{k}\leq C_{k}\lambda_{1},$$where $C_{k}$ is a positive
constant
which depends only on $k$ and in particular is independent of  $(M,g,f)$.
Using an example, K. Funano and T. Shioya showed that the non-negativity of curvature is a necessary condition (see \cite{FS}).
The proof relies on a geometric theory of concentration of metric
measure spaces due to M. Gromov \cite{Gro}. We also note that
A. Hassannezhad demonstrated upper bounds for the eigenvalues without curvature
assumptions \cite{Has}.
\vskip 3mm
In this paper, we consider the following Dirichlet problem of drifting Laplacian:
\begin{equation}
\left\{ \begin{aligned} \label{Eigenvalue-Problem}\Delta_{f}u=-\lambda u,\ \ &{\rm
in}\ \ \ \ \Omega,
         \\
u=0,\ \ &{\rm on}\ \ \partial\Omega,
                          \end{aligned} \right.
                          \end{equation}where
$\Omega\subset M^{n}$ is a bounded domain with piecewise smooth
boundary $\partial\Omega$ in an $n$-dimensional complete metric
measure space $(M^{n},g,e^{-f})$. It is clear that eigenvalue problem \eqref{Eigenvalue-Problem} is exactly eigenvlue \eqref{Eigen-Prob-Lapl} when $f$
is a constant.
If $\lambda_{i}$ is the $i$-th
eigenvalue of this problem,
then the spectrum of the Dirichlet  problem \eqref{Eigenvalue-Problem}  is
discrete and satisfies the following:

\begin{equation*}
0<\lambda_{1}<\lambda_{2}\leq\lambda_{3}\leq\cdots\leq\lambda_{k}\leq\cdots\rightarrow+\infty,
\end{equation*}
where each eigenvalue is repeated according to its multiplicity. All through this paper, we always assume that the dimensional $n$ is larger that one. For this eigenvalue problem, our first result is the
following:

\begin{thm}
\label{thm1.1}Let $(M^{n},g,f)$ be a complete metric measure space, where $M^{n}$ is an $n$-dimensional complete Riemannian manifolds isometrically
immersed in a Euclidean space $\mathbb{R}^{n+p}$, and $\lambda_{i}$ be the $i$-th $(i=1,2,\cdots,k)$
eigenvalue of the Dirichlet problem \eqref{Eigenvalue-Problem}. Then we have

\begin{equation}\label{z1}\begin{aligned}\lambda_{k+1}-\lambda_{k}\leq
C_{n,\Omega,f}k^{\frac{1}{n}},\end{aligned}
\end{equation}
where $C_{n,\Omega,f}$ is a constant dependent on $\Omega$ itself and the dimension $n$.
\end{thm}
In this paper, we also investigate the eigenvalues
of  the closed eigenvalue problem of drifting Laplacian on compact Riemannian manifolds:
\begin{equation}
\label{f-Laplacian-closed} \Delta_{f}u=-\overline{\lambda} u.
\end{equation}
Spectrum of the closed eigenvalue problem \eqref{f-Laplacian-closed}   is
discrete and satisfies
\begin{equation*}
0=\overline{\lambda}_{0}\leq\overline{\lambda}_{1}\leq\overline{\lambda}_{2}\leq\cdots\leq\overline{\lambda}_{k}\leq\cdots\rightarrow+\infty,
\end{equation*}
where each eigenvalue is repeated according to its multiplicity.

Similarly, we assume that $(M^{n},g, f)$ is an $n$-dimensional  closed metric measure space, which is isometrically
immersed in an $(n+p)$-dimensional Euclidean space $\mathbb{R}^{n+p}$, then we have the following:

\begin{thm}
\label{thm1.2}Let $(M^{n},g,f)$ be a closed metric measure space and $M^{n}$ an $n$-dimensional closed Riemannian manifold isometrically immersed into the Euclidean space
$\mathbb{R}^{n+p}$. Assume that $\overline{\lambda}_{i}$ is the $i$-th $(i=1,2,\cdots,k)$
eigenvalue of the closed eigenvalue problem \eqref{f-Laplacian-closed}. Then we have

\begin{equation}\label{z1}\begin{aligned}\overline{\lambda}_{k+1}-\overline{\lambda}_{k}\leq
C_{n,M^{n},f}k^{\frac{1}{n}},\end{aligned}
\end{equation}
where $C_{n,M^{n},f}$ is a constant dependent on $M^{n}$ itself, function $f$, and the dimension $n$.
\end{thm}
\begin{rem}In theorem \ref{thm1.1} and theorem \ref{thm1.2}, the constants $C_{n,\Omega,f}$ and $C_{n,M^{n},f}$ are allowed to be different in different backgrounds.\end{rem}
In 1982, R. S. Hamilton introduced Ricci solitons \cite{Ha1,Ha2},
which are self-similar solutions to the Ricci flow. Because Ricci solitons represent
the fixed points of the Ricci flow, they are an important object in understanding the Ricci flow.
Ricci solitons is an important example of complete metric measure space, which is defined as follows:
 Let $M^{n}$ be a complete Riemannian manifold with
smooth metric $g=(g_{ij})$, then $(M^{n}, g, f)$ is called a gradient
Ricci soliton if there is a constant  $\rho$ such that
\begin{equation}
\label{soliton-equation} R_{ij} +f_{ij} = \rho g_{ij},
\end{equation}
where $R_{ij}$ and $f_{ij}$  denote components of the Ricci tensor and Hessian of $f$, respectively. The Ricci soliton is said to be shrinking, steady and
expanding according as $\rho
> 0$, $\rho = 0$ or $\rho < 0$, respectively.
The function $f$ is called a potential function of the gradient
Ricci soliton (cf. \cite{C}).  From the equation \eqref{soliton-equation}, it is not difficult to see that Ricci solitons are generalizations
of Einteins metrics. We investigate the eigenvalue of the Dirichlet problem
\eqref{Eigenvalue-Problem} of drifting Laplacian on complete noncompact Ricci solitons and prove the following:
\begin{thm}
\label{thm5.5}\label{thm-z-3} Let $(M^{n},g_{ij},f)$ be an $n$-dimensional compact
gradient Ricci Soliton. Then, for any $j$, eigenvalues of the closed
eigenvalue problem \eqref{Eigen-Prob-closed} of drifting Laplacian
satisfy

\begin{equation}\label{z-4-3}\begin{aligned}\overline{\lambda}_{k+1}-\overline{\lambda}_{k}\leq
C_{n,M^{n},f}(k+1)^{\frac{1}{n}},\end{aligned}
\end{equation}
where $$C_{n,M^{n},f}=(\lambda_{1}+c)\sqrt{\frac{32\overline{\alpha}^{2}C_{0}(n)}{n\alpha^{2}+\sum_{j=1}^{n+p}b_{j}}},$$ $C_{0}(n)$ is the same as the one in \eqref{Cheng-Yang-ineq},
\begin{equation*}c=\frac{1}{4}\inf_{\psi\in\Psi}\max_{M^{n}}\left(n^{2}H^{2}+4|\rho f-\rho\overline{c}|
+2\rho f+n\rho-2\rho\overline{c}-S\right),\end{equation*} and
\begin{equation*}\overline{c}=\frac{\int_{M^{n}}fe^{-f}dv}{\int_{M^{n}}e^{-f}dv}.\end{equation*}

\end{thm}
 Let $X:M^{n}\rightarrow\mathbb{R}^{n+p}$ be an $n$-dimensional  submanifold
in  the Euclidean space $\mathbb{R}^{n+p}$. If $X:M^{n}\rightarrow\mathbb{R}^{n+p}$ satisfies $$n\vec H=-X^N,$$ where $\vec H$ and
$X^N$ denote the mean curvature vector and  the orthogonal, then we say that it is called a self-shrinker
projection of $X$  into the normal bundle of $M^n$, respectively.
As another application of the general formula \eqref{gen-for}, we consider the self-shrinker
of the mean curvature flow, which is introduced by G. Huisken in \cite{H}(cf. T. H. Colding
and W. P. Minicozzi \cite{CM}).
\begin{thm}
\label{thm-shrinker} Let $H$ and $X$ denote the mean curvature of $M^n$ and the
position vector of $M^n$, respectively. Then, for  an $n$-dimensional complete self-shrinker $M^{n}$
in  the Euclidean space $\mathbb{R}^{n+p}$, eigenvalues  of the Dirichlet
problem \eqref{Eigenvalue-Problem} of drifting Laplacian with $f=\frac{|X|^2}{2}$ satisfy
\begin{equation*}\begin{aligned}\lambda_{k+1}-\lambda_{k}\leq C_{n,\Omega,X}k^{\frac{1}{n}},\end{aligned}
\end{equation*}
where $$C_{n,\Omega,X}=(\lambda_{1}+c)\sqrt{\frac{32\overline{\alpha}^{2}C_{0}(n)}{n\alpha^{2}+(n+p)\beta}},$$
$$c=\frac{1}{4}\inf_{\psi\in \Psi}\max_{\Omega}\left(n^{2}H^{2}+|2n-|X|^2|+|X|^2\right),$$and $\Psi$ denotes the set of all isometric immersions from $M^n$
into a Euclidean space.
\end{thm}

This paper is organized as follows. In section
\ref{sec2}, we prove several key lemmas. By utilizing those key lemmas,  we prove a general formula
of the eigenvalues of the Dirichlet problem. By the same method, we establish the corresponding general formulas with respect to the closed eigenvalue problem. By
utilizing those general formulas, we give the proofs of
theorem \ref{thm1.1} and theorem \ref{thm1.2} in section \ref{sec3}. In  last part of section \ref{sec3}, we give a gap conjectures of consecutive eigenvalues of the Dirichlet problem \eqref{Eigenvalue-Problem}
of drifting Laplacian on complete Riemannian manifolds. In section \ref{sec4}, we investigate the eigenvalue of the drifting Laplacian on the complete Ricci solitons. As
some further applications, we give the explicit upper bounds for the consecutive eigenvalues of Laplacian on some important Ricci solitons in section \ref{sec5}.
As a further interest, we give the explicit upper bounds for the consecutive eigenvalues of Laplacian on self-shrinkers in section \ref{sec6}. In section \ref{sec7}, we consider the
eigenvalue problem of drifting Laplacian on splitting Riemannian manifolds.
The last section is an appendix, we give the proof of theorem \ref{thm-product-soliton} in this appendix.

\section{General formulas for eigenvalues} \label{sec2}
\vskip3mm

\noindent In this section, we would like to establish  some  general formulas for eigenvalues, which generalizes
a formula of  D. Chen, T. Zheng and H.-C. Yang in \cite{CZY} for the case of Laplacian.  Firstly, we shall use the same notations as in \cite{CZY}.
We define $\mathcal{H}^{\infty}$ by
$$\mathcal{H}^{\infty}=\left\{x=(x_{j})_{j=1}^{\infty}\Bigg{|}x_{j}\in\mathbb{R}\left(\sum^{\infty}_{j=1}x_{j}^{2}\right)^{\frac{1}{2}}<+\infty\right\},$$
with inner product $\langle\cdot,\cdot\rangle_{\infty}$, where $\langle\cdot,\cdot\rangle_{\infty}$ is defined by

$$\langle x,y\rangle_{\infty}=\sum^{\infty}_{j=1}x_{j}y_{j},\ \ \forall x=(x_{j})^{\infty}_{j=1},y=(y_{j})^{\infty}_{j=1}.$$
  Similarly, we can also define  $\mathcal{H}^{2}$ by
$$\mathcal{H}^{2}=\left\{x=(x_{1},x_{2})\Big{|}x_{1},x_{2}\in\mathbb{R}\left( x_{1}^{2}+ x_{2}^{2}\right)^{\frac{1}{2}}<+\infty\right\},$$
with inner product $\langle\cdot,\cdot\rangle_{2}$, where $\langle\cdot,\cdot\rangle_{2}$ is defined by

$$\langle x,y\rangle_{2}=\sum^{2}_{j=1}x_{j}y_{j},\ \ \forall x=(x_{j})^{2}_{j=1},y=(y_{j})^{2}_{j=1}.$$
It is not difficult to see that both $\left(\mathcal{H}^{\infty},\langle\cdot,\cdot,\rangle_{\infty}\right)$ and
 $\left(\mathcal{H}^{2},\langle\cdot,\cdot,\rangle_{2}\right)$ are Hilbert space. The dual space of $\mathcal{H}^{2}$ is denoted by $\left(\mathcal{H}^{2}\right)^{\ast}$. It is well
known that $\left(\mathcal{H}^{2}\right)^{\ast}$ is isomorphic to $\mathcal{H}^{2}$ itself.
By Lagrange multiplier theorem
for real Banach spaces,  D. Chen, T. Zheng and H.-C. Yang proved the following theorem \cite{CZY}:

\begin{thm} \label{thm-c-z-y}Assume that $\{\mu_{j}\}^{\infty}_{j=1}$ is a nondecreasing sequence, i.e.,
$$0 <\mu_{1}\leq\mu_{2}\leq \cdots\leq\mu_{k}\leq\cdots\nearrow \infty,$$
where each $\mu_{i}$ has finite multiplicity $\mu_{i}$ and is repeated according to its multiplicity.
Define
\begin{equation}\begin{aligned}\label{2-2} &B =
\sum^{\infty}_{j=1}
x^{2}_{j}>0,\\
&A =
\sum^{\infty}_{j=1}
\mu^{2}_{j}x^{2}_{j}, x=(x_{j})^{\infty}_{j=1}\in\mathcal{H}^{\infty}.\end{aligned} \end{equation}
If $$x_{m_{1}}\neq0$$ and $$\sum^{\infty}_{j=1}\mu_{j}x^{2}_{j}<\sqrt{AB},$$ under the conditions in \eqref{2-2}, we have
\begin{equation}\begin{aligned}\label{2-3}
\sum^{\infty}_{j=1}
\mu_{j}x^{2}_{j}\leq \frac{A+\mu_{m_{1}}\mu_{m_{1}+1}B}{
\mu_{m_{1}} + \mu_{m_{1}+1}}.\end{aligned} \end{equation}\end{thm}
Next, we complete the proof of the general formula by using the  same method as in  D. Chen, T. Zheng and H.-C. Yang \cite{CZY}. For the convenience of readers, we shall give a
self contained proof.\begin{lem}\label{lem2.1}
Let $(M^{n},g,f)$ be an $n$-dimensional complete metric measure space and $\Omega$ a bounded domain
with piecewise smooth boundary $\partial\Omega$ on the Riemannian manifold $M^{n}$.
Assume that $\lambda_{i}$  is  the $i^{\text{th}}$ eigenvalue of the
Dirichlet problem \eqref{Eigenvalue-Problem}  and $u_{i}$ is an
orthonormal eigenfunction corresponding to $\lambda_{i}$, $i =
1,2,\cdots$, such that

\begin{equation*}
\left\{ \begin{aligned}\Delta_{f} u_{i}=-\lambda u_{i},\ \ \ \ \ \ \ &
in\ \ \ \ \Omega,
         \\ u_{i}=0,\ \ \ \ \ \ \ \ \ \ \ \ \ \ &  on\ \ \partial\Omega,
\\
\int_{\Omega}
u_{i}u_{j}e^{-f}dv=\delta_{ij},\ \ & for~any  \ i,j=1,2,\cdots.
                          \end{aligned} \right.
                          \end{equation*}
Then, for any function $h(x)\in C^{3}(\Omega)\cap C^{2}(\overline{\Omega})$ and any
integer $k,i \in \mathbb{Z}^{+},~(k>i\geq1)$,  eigenvalues of the Dirichlet problem \eqref{Eigenvalue-Problem} satisfy

\begin{equation}\begin{aligned}\label{general-formula-1}((\lambda_{k+2}-\lambda_{i})+(\lambda_{k+1}-\lambda_{i}))\|\nabla hu_{i}\|^{2}_{\Omega}+\sum^{k}_{j=1}
(\lambda_{j}-\lambda_{i})^{2}\|hu_{i}\|^{2}_{\Omega}
\leq\|2\langle\nabla h,\nabla u_{i}\rangle+u_{i}\Delta_{f}h\|^{2}_{\Omega},
\end{aligned}\end{equation}
where
\begin{equation*}
\|h(x)\|_{\Omega}^{2} =\int_{\Omega}h^{2}(x)e^{-f}dv.
\end{equation*}
\end{lem}

\begin{proof} Since  $u_{j}$ is  an
orthonormal eigenfunction corresponding to the eigenvalue
$\lambda_j$,
 $\{u_{j}\}^{\infty}_{j=1}$ forms an orthonormal basis of
the weighted $L^{2}(\Omega)$. From the Rayleigh-Ritz inequality \cite{Chavel}, we
have
\begin{equation}
\begin{aligned}
\label{2.12} \lambda_{k+1}\leq-\frac{\displaystyle
\int_{\Omega}\varphi\Delta_{f}\varphi e^{-f}dv}{\displaystyle
\int_{\Omega}\varphi^{2}e^{-f}dv},
\end{aligned}
\end{equation}
for any function $\varphi$ satisfies $$\int_{\Omega}\varphi
u_{j}e^{-f}dv=0 ,\ \ 1\leq j\leq k.$$ Putting
\begin{equation*}
a_{ij}=\int_{\Omega}hu_{i}u_{j}e^{-f}dv
\end{equation*}
and
\begin{equation*}
\varphi_{i}=hu_{i}-\sum^{k}_{j=1}a_{ij}u_{j},
\end{equation*}
then, we have
\begin{equation}\label{a}
a_{ij}=a_{ji}.
\end{equation}
By a simple calculation, we find that
\begin{equation}
\begin{aligned}
\label{2.13} \int_{\Omega}\varphi_{i}u_{l}e^{-f}dv&=0,
\end{aligned}
\end{equation}
for $1\leq i, l\leq k.$  \eqref{2.12} implies
\begin{equation*}
\lambda_{k+1}\leq-\frac{\displaystyle
\int_{\Omega}\varphi_{i}\Delta_{f}\varphi_{i}e^{-f}dv}{\displaystyle
\int_{\Omega}\varphi^{2}_{i}e^{-f}dv}.
\end{equation*}
By defining
$$
b_{ij}=-\int_{\Omega}(u_{j}\Delta_{f}h+2\langle\nabla h,\nabla
u_{j}\rangle) u_{i}e^{-f}dv,
$$
we have
\begin{equation}
\label{a-b}b_{ij}=(\lambda_{i}-\lambda_{j})a_{ij}.
\end{equation}
From the Stokes' theorem, we have
\begin{equation}-2\int_{\Omega}hu_{i}\langle\nabla \overline{h},\nabla u_{i}\rangle =-
\int_{\Omega}h\langle\nabla \overline{h},\nabla u^{2}_{i}\rangle
=\int_{\Omega}(\langle\nabla h,\nabla \overline{h}\rangle+ h\Delta_{f} \overline{h})u^{2}_{i}.
\end{equation}
From \eqref{self-adjiont}, \eqref{a} and \eqref{a-b}, we deduce that
\begin{equation}\label{ineq-a-b}
\begin{aligned}
\int_{\Omega}\overline{\varphi_{i}}\Delta _{f}\varphi_{i}e^{-f}dv&=\int_{\Omega}\overline{\varphi_{i}}\Delta
_{f}\Bigg{(}hu_{i}-\sum^{k}_{j=1}a_{ij}u_{j}\Bigg{)}e^{-f}dv
\\&=\int_{\Omega}\overline{\varphi_{i}}\Bigg{(}\Delta _{f}(hu_{i})-\Delta _{f}\left(\sum^{k}_{j=1}a_{ij}u_{j}\right)\Bigg{)}e^{-f}dv
\\&=u_{i}\Delta_{f}h-\lambda_{i}hu_{i}+2\langle\nabla h,\nabla u_{j}\rangle
+\sum^{k}_{j=1}\lambda_{j}a_{ij}u_{j},
\\&=\sum^{\infty}_{j=k+1}\overline{a_{ij}}b_{ij}-\lambda_{i}\sum^{\infty}_{j=k+1}|a_{ij}|^{2}
\\&=\frac{1}{2}\sum^{\infty}_{j=k+1}(\lambda_{i}-\lambda_{j})|a_{ij}|^{2}-\lambda_{i}\sum^{\infty}_{j=k+1}|a_{ij}|^{2}
\end{aligned}
\end{equation}
From the Rayleigh-Ritz inequality (cf. \cite{Chavel}) and \eqref{ineq-a-b}, we have

\begin{equation*}\lambda_{k+1}\leq-\frac{\int_{\Omega}\overline{\varphi_{i}}\Delta\varphi_{i}e^{-f}dv}{\int_{\Omega}|\varphi_{i}|^{2}e^{-f}dv}
=\frac{\sum^{\infty}_{j=k+1}(\lambda_{i}-\lambda_{j})|a_{ij}|^{2}}{\lambda_{i}\sum^{\infty}_{j=k+1}|a_{ij}|^{2}}+\lambda_{i}\end{equation*}
i.e.

\begin{equation}(\lambda_{k+1}-\lambda_{i})
\sum^{\infty}_{j=k+1}|a_{ij}|^{2}\leq\sum^{\infty}_{j=k+1}
(\lambda_{j}-\lambda_{i})|a_{ij}|^{2}.\end{equation}
Utilizing the Cauchy-Schwarz inequality, we yield

\begin{equation*}
\left(\sum^{\infty}_{j=k+1}(\lambda_{j}-\lambda_{i})|a_{ij}|^{2}\right)^{2}\leq\sum^{\infty}_{j=k+1}
(\lambda_{j}-\lambda_{i})^{2}|a_{ij}|^{2}\sum^{\infty}_{j=k+1}
|a_{ij}|^{2},\end{equation*}
which is equivalent to the following:

\begin{equation*}\begin{aligned}
\left(\|\nabla hu_{i}\|^{2}_{\Omega}-\sum^{k}_{j=1}(\lambda_{j}-\lambda_{i})|a_{ij}|^{2}\right)^{2}
&\leq\left(\|hu_{i}\|^{2}_{\Omega}-\sum^{k}_{j=1}
|a_{ij}|^{2}\right)\left(\|2\langle\nabla h,\nabla u_{i}\rangle+u_{i}\Delta_{f}h\|^{2}_{\Omega}-\sum^{k}_{j=1}
(\lambda_{j}-\lambda_{i})^{2}|a_{ij}|^{2}\right).\end{aligned}\end{equation*}
Define

\begin{equation*}\begin{aligned}\mathcal{A}(i)&=\|2\langle\nabla h,\nabla u_{i}\rangle+u_{i}\Delta_{f}h\|^{2}_{\Omega}-\sum^{k}_{j=1}
(\lambda_{j}-\lambda_{i})^{2}|a_{ij}|^{2}\\& =\sum^{\infty}_{j=k+1}
(\lambda_{j}-\lambda_{i})^{2}|a_{ij}|^{2}\geq0;\end{aligned}\end{equation*}

$$\mathcal{B}(i)=\|hu_{i}\|^{2}_{\Omega}-\sum^{k}_{j=1}
|a_{ij}|^{2} =\sum^{\infty}_{j=k+1}|a_{ij}|^{2},\ \ {\rm here }\ \ \int_{\Omega}hu_{i}u_{k+1}e^{-f}dv\neq0;$$and

$$\mathcal{C}(i)=\|\nabla hu_{i}\|^{2}_{\Omega}
-\sum^{k}_{j=1}(\lambda_{j}-\lambda_{i})|a_{ij}|^{2}=\sum^{\infty}_{j=k+1}(\lambda_{j}-\lambda_{i})|a_{ij}|^{2}.$$
Since $hu_{i}$ is not the $\mathbb{C}$-linear combination of
$u_{1},\cdots,u_{k+1},$
there exists some $l > k+1$ such that
$$a_{l}=\int_{\Omega}hu_{i}u_{l}e^{-f}dv\neq0.$$
It is not difficult to see that
$$\lambda_{i} < \lambda_{k+1} < \lambda_{k+2}\leq\lambda_{l},$$ therefore,
the vector

$$\left(|a_{ij}|\right)^{\infty}_{j=k+1}$$
is not proportional to
$$\left((\lambda _{j}-\lambda _{i})^{2}|a_{ij}|\right)^{\infty}_{j=k+1}.$$
From the Cauchy-Schwarz inequality, we have

\begin{equation}
\label{(3-10)} \mathcal{C}(i)<\sqrt{\mathcal{A}(i)\mathcal{B}(i)}\end{equation}
Since $a_{k+1}\neq0$, from \eqref{(3-10)} and  theorem \ref{thm-c-z-y}, we obtain

\begin{equation}\label{C}\mathcal{C}(i)\leq\frac{\mathcal{A}(i)+(\lambda_{k+2}-\lambda_{i})(\lambda_{k+1}-\lambda_{i})\mathcal{B}(i)}
{(\lambda_{k+2}-\lambda_{i})-(\lambda_{k+1}-\lambda_{i})}\end{equation}
From \eqref{C}, and the definition of $\mathcal{A}(i)$, $\mathcal{B}(i)$ and $\mathcal{C}(i)$, one can infer that

\begin{equation*}\begin{aligned}&((\lambda_{k+2}-\lambda_{i})+(\lambda_{k+1}-\lambda_{i}))\|\nabla hu_{i}\|^{2}_{\Omega}
\\&\leq\|2\langle\nabla h,\nabla u_{i}\rangle+u_{i}\Delta_{f}h\|^{2}_{\Omega}-\sum^{k}_{j=1}
(\lambda_{j}-\lambda_{i})^{2}\|hu_{i}\|^{2}_{\Omega}
-(\lambda_{k+2}-\lambda_{j})(\lambda_{k+1}-\lambda_{j})|a_{ij}|^{2}
\\&\leq\|2\langle\nabla h,\nabla u_{i}\rangle+u_{i}\Delta_{f}h\|^{2}_{\Omega}-\sum^{k}_{j=1}
(\lambda_{j}-\lambda_{i})^{2}\|hu_{i}\|^{2}_{\Omega}.
\end{aligned}\end{equation*}
Therefore, we complete the proof of this lemma.\end{proof}

\begin{corr}\label{coro2.2}Under the assumption of the lemma \ref{lem2.1}, for any real value function $F\in C^{3}(\Omega)\cap C^{2}(\overline{\Omega}),$  we have
\begin{equation}\begin{aligned}\label{general-formula-2}((\lambda_{k+2}-\lambda_{i})+(\lambda_{k+1}-\lambda_{i}))\|\nabla Fu_{i}\|^{2}_{\Omega}
&\leq2\sqrt{(\lambda_{k+2}-\lambda_{i})(\lambda_{k+1}-\lambda_{i})\||\nabla F|^{2}u_{i}\|^{2}_{\Omega}}
\\&\ \ \ \ \ \ \ +\|2\langle\nabla F,\nabla u_{i}\rangle+u_{i}\Delta_{f}F\|^{2}_{\Omega}.
\end{aligned}\end{equation}
\end{corr}
\begin{proof}
Taking $h=e^{i\alpha F}$, $F\in\mathbb{R}\backslash\{0\}$ in
\eqref{general-formula-1}, we obtain

\begin{equation}\begin{aligned}\label{lk2-li1}\eta^{2}((\lambda_{k+2}-\lambda_{i})+(\lambda_{k+1}-\lambda_{i}))\|\nabla Fu_{i}\|^{2}_{\Omega}
&\leq\eta^{4}\||\nabla F|^{2}u_{i}\|^{2}_{\Omega}+\eta^{2}\|2\langle\nabla F,\nabla u_{i}\rangle\\&\ \ \ \ \ \ \ +u_{i}\Delta_{f}F\|^{2}_{\Omega}+(\lambda_{k+2}-\lambda_{i})(\lambda_{k+1}-\lambda_{i}).
\end{aligned}\end{equation}
From \eqref{lk2-li1}, we deduce

\begin{equation}\begin{aligned}\label{lk2-l12}((\lambda_{k+2}-\lambda_{i})+(\lambda_{k+1}-\lambda_{i}))\|\nabla Fu_{i}\|^{2}_{\Omega}
&\leq\eta^{2}\||\nabla F|^{2}u_{i}\|^{2}_{\Omega}+\|2\langle\nabla F,\nabla u_{i}\rangle+u_{i}\Delta_{f}F\|^{2}_{\Omega}\\&\ \ \ \ \ \ \ +\frac{1}{\eta^{2}}(\lambda_{k+2}-\lambda_{i})(\lambda_{k+1}-\lambda_{i}).
\end{aligned}\end{equation}
Using the Cauchy-Schwarz inequality in \eqref{lk2-l12}, we yield \eqref{general-formula-2}. This finishes the proof of this lemma.
\end{proof}

By utilizing corollary \ref{coro2.2}, we have

\begin{prop}\label{prop2.3}Let $\tau$ be a constant such that, for any $i=1,2,\cdots,k,$ $\lambda_{i}+\tau>0$.
Under the assumption of the lemma \ref{lem2.1}, for any $j=1,2,\cdots,l,$ and any real value function $F_{j}\in C^{3}(\Omega)\cap C^{2}(\overline{\Omega})$,  we have

\begin{equation}\begin{aligned}\label{general-formula-2}\sum_{j=1}^{l}\frac{a_{j}^{2}+b_{j}}{2}\left(\lambda_{k+2}-\lambda_{k+1}\right)^{2}
\leq4(\lambda_{k+2}+\tau)\sum_{j=1}^{l}\|2\langle\nabla F_{j},\nabla u_{i}\rangle+u_{i}\Delta_{f}F_{j}\|_{\Omega}^{2}
,
\end{aligned}\end{equation}
where

\begin{equation*}a_{j}=\sqrt{\|\nabla F_{j}u_{i}\|_{\Omega}^{2}},\end{equation*}

\begin{equation*}b_{j}=\sqrt{\||\nabla F_{j}|^{2}u_{i}\|_{\Omega}^{2}},\end{equation*}
\begin{equation}a_{j}^{2}\geq b_{j},\end{equation}  and \begin{equation*}
\|F(x)\|=\int_{\Omega}F(x)e^{-f}dv.
\end{equation*}

\end{prop}

\begin{proof}By the assumption, we have

\begin{equation*}\begin{aligned}\frac{a_{j}^{2}-b_{j}}{2}
\left(\sqrt{\lambda_{k+2}-\lambda_{i}}+\sqrt{\lambda_{k+1}-\lambda_{i}}\right)^{2}\geq0\end{aligned}\end{equation*}
which is equivalent to the following:

\begin{equation}\begin{aligned}\label{gap-ineq}&a_{j}^{2}((\lambda_{k+2}-\lambda_{i})+(\lambda_{k+1}-\lambda_{i}))
-2b_{j}\sqrt{(\lambda_{k+2}-\lambda_{i})(\lambda_{k+1}-\lambda_{i})}\\&\geq\frac{a_{j}^{2}+b_{j}}{2}
\left(\sqrt{\lambda_{k+2}-\lambda_{i}}-\sqrt{\lambda_{k+1}-\lambda_{i}}\right)^{2}.\end{aligned}\end{equation}
By \eqref{gap-ineq} and \eqref{general-formula-1}, we have
\begin{equation*}\begin{aligned}\frac{a_{j}^{2}+b_{j}}{2}
\left(\sqrt{\lambda_{k+2}-\lambda_{i}}-\sqrt{\lambda_{k+1}-\lambda_{i}}\right)^{2}
&\leq\|2\langle\nabla h_{j},\nabla u_{i}\rangle+u_{i}\Delta h_{j}\|_{\Omega}^{2}.
\end{aligned}\end{equation*}
Taking sum over $j$ from $1$ to $l$, we yield

\begin{equation}\begin{aligned}\label{gen-for-4}\sum_{j=1}^{l}\frac{a_{j}^{2}+b_{j}}{2}\left(\sqrt{\lambda_{k+2}-\lambda_{i}}-\sqrt{\lambda_{k+1}-\lambda_{i}}\right)^{2}
 \leq\sum_{j=1}^{l}\|2\langle\nabla F_{j},\nabla u_{i}\rangle+u_{i}\Delta_{f} F_{j}\|_{\Omega}^{2}.
\end{aligned}\end{equation}
Multiplying \eqref{gen-for-4} by
$\left(\sqrt{\lambda_{k+2}-\lambda_{i}}+\sqrt{\lambda_{k+1}-\lambda_{i}}\right)^{2}$ on both sides, one can infer that

\begin{equation*}\begin{aligned}\sum_{j=1}^{l}\frac{a_{j}^{2}+b_{j}}{2}\left(\lambda_{k+2}-\lambda_{k+1}\right)^{2}
&\leq\sum_{j=1}^{l}\|2\langle\nabla F_{j},\nabla u_{i}\rangle+u_{i}\Delta F_{j}\|_{\Omega}^{2}\left(\sqrt{\lambda_{k+2}-\lambda_{i}}+\sqrt{\lambda_{k+1}-\lambda_{i}}\right)^{2}\\
&=\sum_{j=1}^{l}\|2\langle\nabla F_{j},\nabla u_{i}\rangle+u_{i}\Delta F_{j}\|_{\Omega}^{2}\\&\times\left(\sqrt{(\lambda_{k+2}+\tau)
-(\lambda_{i}+\tau)}+\sqrt{(\lambda_{k+1}+\tau)-(\lambda_{i}+\tau)}\right)^{2}\\
&\leq4(\lambda_{k+2}+\tau)\sum_{j=1}^{l}\|2\langle\nabla F_{j},\nabla u_{i}\rangle+u_{i}\Delta F_{j}\|_{\Omega}^{2}
.
\end{aligned}\end{equation*}
which is the inequality \eqref{general-formula-2}.
Therefore, we finish the proof of this proposition.

\end{proof}

By the same method as the proof of proposition \ref{prop2.3}, one can prove the following proposition if one notices to count the number of eigenvalues from $0$.

\begin{prop}\label{gen-for-2}Let $(M^{n},g)$ be an $n$-dimensional closed metric measure space.
Assume that $\overline{\lambda}_{i}$  is  the $i^{\text{th}}$ eigenvalue of the
eigenvalue problem \eqref{Eigen-Prob-closed}  and $u_{i}$ is an
orthonormal eigenfunction corresponding to $\overline{\lambda}_{i}$, $i =0,
1,2,\cdots$, such that
\begin{equation*}
\left\{ \begin{aligned}\Delta_{f} u_{i}=-\lambda u_{i},\ \ \ \ \ \ \ &
in\ \ \ \ M^{n},
\\
\int_{M^{n}}
u_{i}u_{j}e^{-f}dv=\delta_{ij},\ \ & for~any  \ i,j=0,1,2,\cdots.
                          \end{aligned} \right.
                          \end{equation*}Let $\tau$ be a constant such that, for any $i=0,1,2,\cdots,k,$ $\overline{\lambda}_{i}+\tau>0$.
Then, for any $j=0,1,2,\cdots,l,$ and any real value function $h_{j}\in  C^{2}(M^{n})$,  we have

\begin{equation}\begin{aligned}\label{general-formula-4}\sum_{j=1}^{l}\frac{a_{j}^{2}+b_{j}}{2}\left(\overline{\lambda}_{k+2}-\overline{\lambda}_{k+1}\right)^{2}
\leq4(\overline{\lambda}_{k+2}+\tau)\sum_{j=1}^{l}\|2\langle\nabla F_{j},\nabla u_{i}\rangle+u_{i}\Delta_{f}F_{j}\|_{M^{n}}^{2}
,
\end{aligned}\end{equation}
where

\begin{equation*}a_{j}=\sqrt{\|\nabla F_{j}u_{i}\|_{M^{n}}^{2}},\end{equation*}

\begin{equation*}b_{j}=\sqrt{\||\nabla F_{j}|^{2}u_{i}\|_{M^{n}}^{2}},\end{equation*} $$\|h\|^{2}\int_{M^{n}}h^{2}e^{-f}dv,$$ and
\begin{equation*}a_{j}^{2}\geq b_{j}.
\end{equation*}

\end{prop}
Next, we state the general formula given by C. Xia and H. Xu in \cite{XX}, which will be used in the next section.
\begin{prop}\label{thm-x-x}
Let $(M^{n},g)$ be an $n$-dimensional complete metric measure space.
Assume that $\lambda_{i}$  is  the $i^{\text{th}}$ eigenvalue of the Dirichlet problem \eqref{Eigenvalue-Problem}  and $u_{i}$ is an
orthonormal eigenfunction corresponding to $\lambda_{i}$, $i =1,2,\cdots$, such that
\begin{equation*}
\Delta_{f}u_{i} =-\lambda_{i}u_{i}\ \  and\ \ \int_{\Omega}
u_{i}u_{j}e^{-f}dv=\delta_{ij},\ \
 \text{\rm for any} \ i,j=1,2,\cdots.
\end{equation*}
Then, for any function $h(x)\in C^{2}(\Omega)$ and any positive
integer $k$,  eigenvalues of the Dirichlet problem \eqref{Eigenvalue-Problem} satisfy
\begin{equation}
\begin{aligned}
\label{gen-for}
\sum^{k}_{i=1}(\lambda_{k+1}-\lambda_{i})^{2}\|u_{i}\nabla
h\|_{\Omega}^{2} \leq\sum^{k}_{i=1}(\lambda_{k+1}-\lambda_{i})
\|2\langle\nabla h,~\nabla u_{i}\rangle+u_{i}\Delta
_{f}h\|_{\Omega}^{2}.
\end{aligned}
\end{equation}
\end{prop}Similarly, Q.-M. Cheng and the author \cite{CZ} proved the following (also see \cite{Z}):
\begin{prop}\label{prop2.7}
Let $(M^{n},g,)$ be an $n$-dimensional closed metric measure space.
Assume that $\overline{\lambda}_{i}$  is  the $i^{\text{th}}$ eigenvalue of the
close eigenvalue problem \eqref{Eigen-Prob-closed}  and $u_{i}$ is an
orthonormal eigenfunction corresponding to $\overline{\lambda}_{i}$, $i =0,
1,2,\cdots$, such that
\begin{equation*}
\Delta_{f}u_{i} =-\overline{\lambda}_{i}u_{i}\ \  and\ \ \int_{M^{n}}
u_{i}u_{j}e^{-f}dv=\delta_{ij},\ \
 \text{\rm for any} \ i,j=0,1,2,\cdots.
\end{equation*}
Then, for any function $h(x)\in C^{2}(M^{n})$ and any positive
integer $k$,  eigenvalues of the close eigenvalue problem \eqref{Eigen-Prob-closed} satisfy
\begin{equation}
\begin{aligned}
\label{2.11}
\sum^{k}_{i=0}(\overline{\lambda}_{k+1}-\overline{\lambda}_{i})^{2}\|u_{i}\nabla
h\|_{M^{n}}^{2} \leq\sum^{k}_{i=0}(\overline{\lambda}_{k+1}-\overline{\lambda}_{i})
\|2\langle\nabla h,~\nabla u_{i}\rangle+u_{i}\Delta
_{f}h\|_{M^{n}}^{2}.
\end{aligned}
\end{equation}

\end{prop}
\begin{rem}However, it is very well known that the drifting Laplacian
$\Delta f :=\Delta+ \langle \nabla f,\cdot\rangle$
on a compact measure metric space $(M^{n}, g, e^{-f} dv)$ is unitarily equivalent
to the Schr\"{o}dinger operator $$\Delta + \frac{1}{2}\Delta f +\frac{1}{4}|\nabla f|^{2}$$
on $(M^{n}, g)$ and thus it has
the same spectrum (see for instance \cite{S}). Therefore,  proposition \ref{prop2.7} can be proved  by the similar method  in \cite{SI,SHI,IM1,IM2,WX}
by replacing the potential $q$ or $V$
in that papers by
$\Delta + \frac{1}{2}\Delta f +\frac{1}{4}|\nabla f|^{2}$.
\end{rem}

\section{Proofs of theorem \ref{thm1.1} and theorem \ref{thm1.2}}\label{sec3}
\vskip 3mm

In this section, we would like to give the proofs of theorem \ref{thm1.1} and theorem \ref{thm1.2}. Firstly, we need the
following lemma which will be
found in \cite{CC}.

\begin{lem}\label{lem3.1}
For  an $n$-dimensional  submanifold  $M^{n}$ in Euclidean space
$\mathbb{R}^{n+p}$,   let $y=(y^{1},y^{2},\cdots,y^{n+p})$ is the
position vector of a point $p\in M^{n}$ with
$y^{\alpha}=y^{\alpha}(x_{1}, \cdots, x_{n})$, $1\leq \alpha\leq
n+p$, where $(x_{1}, \cdots, x_{n})$ denotes a local coordinate
system of $M^n$. Then, we have
\begin{equation}\label{lemma-3.1-1}
\sum^{n+p}_{\alpha=1}g(\nabla y^{\alpha},\nabla y^{\alpha})= n,
\end{equation}
\begin{equation}
\begin{aligned}\label{lemma-3.1-2}
\sum^{n+p}_{\alpha=1}g(\nabla y^{\alpha},\nabla u)g(\nabla
y^{\alpha},\nabla w)=g(\nabla u,\nabla w),
\end{aligned}
\end{equation}
for any functions  $u, w\in C^{1}(M^{n})$,
\begin{equation}\label{lemma-3.1-3}
\begin{aligned}
\sum^{n+p}_{\alpha=1}(\Delta y^{\alpha})^{2}=n^{2}H^{2},
\end{aligned}
\end{equation}
\begin{equation}\label{lemma-3.1-4}
\begin{aligned}
\sum^{n+p}_{\alpha=1}\Delta y^{\alpha}\nabla y^{\alpha}= 0,
\end{aligned}
\end{equation}
where $H$ is the mean curvature of $M^{n}$.
\end{lem}

Utilizing the general formula \eqref{gen-for}, one can deduce that

\begin{equation}
\begin{aligned}
\label{Yang-type-ineq} \sum^{k}_{i=1}(\lambda_{k+1}-\lambda_{i})^{2}
&\leq\dfrac4n\sum^{k}_{i=1}(\lambda_{k+1}-\lambda_{i}) \left(\lambda_i+\int_{\Omega}(n^{2}H^{2}+2\Delta
f-|\nabla f|^{2})e^{-f}dv\right)\\&\leq\dfrac4n\sum^{k}_{i=1}(\lambda_{k+1}-\lambda_{i}) (\lambda_i+c),
\end{aligned}
\end{equation}where
\begin{equation}\label{c}
c:=\frac{1}{4}\inf_{\psi\in \Psi}\max_{M^{n}}(n^{2}H^{2}+2|\Delta_{f}
f|+|\nabla f|^{2}),
\end{equation} and  $\Psi$ denotes the set of all isometric immersions from $M^n$
into a Euclidean space.

{\bf A recursion formula of Cheng and Yang}.  Let  $\mu_1 \leq  \mu_2 \leq  \dots,
\leq \mu_{k+1}$ be any positive  real numbers satisfying
\begin{equation*}
  \sum_{i=1}^k(\mu_{k+1}-\mu_i)^2 \le
 \frac 4n\sum_{i=1}^k\mu_i(\mu_{k+1} -\mu_i).
\end{equation*}
 Define
 \begin{equation*}
 \Lambda_k=\frac 1k\sum_{i=1}^k\mu_i,\qquad T_k=\frac 1k
\sum_{i=1}^k\mu_i^2, \ \ \
F_k=\left (1+\frac 2n \right )\Lambda_k^2-T_k.
\end{equation*}
Then, we have
\begin{equation}
F_{k+1}\leq C(n,k) \left ( \frac {k+1}k \right )^{\frac 4n}F_k,
\end{equation}
where
$$
C(n,k) =1-\frac 1{3n}
  \left (\frac k{k+1}\right )^{\frac
  4n}\frac {\left(1+\frac 2n\right )\left (1+
  \frac 4n\right )}{(k+1)^3}<1.
$$
By using the recursion formula given by Q.-M. Cheng and H.-C. Yang, the author proved the following
\begin{thm}
Let $(M^{n},g,f)$ be an $n$-dimensional complete  metric measure space. $\lambda_{k}$ denotes the
$k$-th eigenvalue of the Dirichlet
problem \eqref{Eigenvalue-Problem} of  the drifting Laplacian. Then, for
any $k\geq 1$,
\begin{equation}\label{zlz-upper}
\lambda_{k+1}+ c \leq (1+\frac {4}{n})(\lambda_{1}+c)k^{2/n},
\end{equation}
where $c$ is the same constant as in \eqref{c}.
\end{thm}

\textbf{\emph{Proof of theorem}} \ref{thm1.1}. Let $F_{j}(x)=\alpha_{j}x^{j}$ and $a_{j}>0$, such that

\begin{equation*}a_{j}^{2}=\|\nabla F_{j}u_{i}\|_{\Omega}^{2}\geq\sqrt{\||\nabla F_{j}|^{2}u_{i}\|_{\Omega}^{2}}=b_{j}\geq0,\end{equation*}
\begin{equation}\begin{aligned}\label{sum-01} \sum_{j=1}^{n+p}\int2 u_{i}\langle\nabla F_{j},\nabla f\rangle\Delta  F_{j}e^{-f}dv=0,\end{aligned}\end{equation}and
\begin{equation}\begin{aligned}\label{sum-0} \sum_{j=1}^{n+p}\int2 u_{i}\langle\nabla F_{j},\nabla u_{i}\rangle\Delta  F_{j}e^{-f}dv=0,\end{aligned}\end{equation}
where   $j=1,2,\cdots,n+p,$ and $x^{j}$ denotes the $j$-th standard coordinate function of the Euclidean space $\mathbb{R}^{n+p}$.
Let $$\alpha=\min_{1\leq j\leq n+p}\{\alpha_{j}\},$$$$\overline{\alpha}=\max_{1\leq j\leq n+p}\{\alpha_{j}\},$$$$\beta=\min_{1\leq j\leq n+p}\{b_{j}\},$$
and $l=n+p$, then, by lemma \ref{lem2.1} and \eqref{lemma-3.1-1}, we have

\begin{equation}\begin{aligned}\sum_{j=1}^{l}\frac{a_{j}^{2}+b_{j}}{2}&=\sum_{j=1}^{n+p}\frac{a_{j}^{2}+b_{j}}{2}\\&\geq\frac{1}{2}\left(n\alpha^{2}+
\sum_{j=1}^{n+p}b_{j}\right)\\&\geq\frac{1}{2}\left(n\alpha^{2}+(n+p)\beta\right),\end{aligned}\end{equation}
and
 \begin{equation}\begin{aligned}\label{3.10} \sum^{n+p}_{j=1}\int_{\Omega}u^{2}_{i}
 \langle\nabla F_{j},\nabla f\rangle^{2} e^{-f}dv= \sum^{n+p}_{j=1}\int_{\Omega}u^{2}_{i} \langle\nabla\left(a_{j}x^{j}\right),\nabla f\rangle^{2} e^{-f}dv
\leq\overline{\alpha}^{2}\int_{\Omega}u^{2}_{i}|\nabla f|^{2}e^{-f}dv.\end{aligned}\end{equation} From \eqref{lemma-3.1-3}, we obtain
\begin{equation}\sum^{n+p}_{j=1}(\Delta F_{j})^{2}=\sum^{n+p}_{j=1}\left(\Delta \left(a_{j}x^{j}\right)\right)^{2}\leq\overline{\alpha}^{2}n^{2}H^{2}.\end{equation}
Utilizing \eqref{lemma-3.1-2} and \eqref{self-adjiont}, we have
\begin{equation}\begin{aligned}-4\sum^{n+p}_{j=1}\int_{\Omega}u_{i}\langle\nabla F_{j},\nabla u_{i}\rangle\langle\nabla F_{j},\nabla f\rangle e^{-f}dv
&=-4\sum^{n+p}_{j=1}\int_{\Omega}u_{i}\langle\nabla \left(a_{j}x^{j}\right),\nabla u_{i}\rangle\langle\nabla \left(a_{j}x^{j}\right),\nabla f\rangle e^{-f}dv\\
&\leq2\overline{\alpha}^{2}|\int_{\Omega}\langle\nabla f,\nabla u^{2}_{i}\rangle e^{-f}dv|\\
&=2\overline{\alpha}^{2}|\int_{\Omega}u^{2}_{i}\Delta_{f}f e^{-f}dv|,\end{aligned}\end{equation}
and

\begin{equation}\begin{aligned}\label{3.13}\sum_{j=1}^{n+p}\int_{\Omega}\langle\nabla F_{j},\nabla u_{i}\rangle^{2}e^{-f}dv
&=\sum_{j=1}^{n+p}\int_{\Omega}\langle\nabla \left(a_{j}x^{j}\right),\nabla u_{i}\rangle^{2}e^{-f}dv\\
&\leq\overline{\alpha}^{2}
\sum_{j=1}^{n+p}\int_{\Omega}\langle\nabla x_{j},\nabla u_{i}\rangle^{2}e^{-f}dv\\
&=\overline{\alpha}^{2}\lambda_{i}.\end{aligned}\end{equation}
Let $\Psi$ denote  the set of all isometric immersions from $M^n$
into a Euclidean space. Define
$$
c=\frac{1}{4}\inf_{\psi\in \Psi}\max_{\Omega}\left(|\nabla f|^{2}+2|\Delta_{f}f|
+n^{2}H^{2}\right)>0.
$$ Since eigenvalues are invariant under isometries,
by lemma \ref{lem3.1}, \eqref{sum-01}, \eqref{sum-0}, and \eqref{3.10}-\eqref{3.13}, we have

\begin{equation}\begin{aligned}&4(\lambda_{k+2}+c)\sum_{j=1}^{n+p}\|2\langle\nabla F_{j},\nabla u_{i}\rangle+u_{i}\Delta_{f} F_{j}\|_{\Omega}^{2}\\&\leq4(\lambda_{k+2}+c)\overline{\alpha}^{2}\left(4\lambda_{i}+\int_{\Omega}u^{2}_{i}|\nabla f|^{2}e^{-f}dv+2|\int_{\Omega}u^{2}_{i}\Delta_{f}f e^{-f}dv|+\int_{\Omega}u_{i}^{2}n^{2}H^{2}e^{-f}dv\right)
\\&\leq4(\lambda_{k+2}+c)\overline{\alpha}^{2}\left(4\lambda_{i}+\int_{\Omega}u^{2}_{i}\left(|\nabla f|^{2}+2|\Delta_{f}f|
+n^{2}H^{2}\right)e^{-f}dv\right)
\\&\leq16\lambda_{k+2}\overline{\alpha}^{2}\left(\lambda_{i}+c\right).\end{aligned}\end{equation}
In proposition \ref{prop2.3}, we let $i=1,\tau=c$. Then, from \eqref{general-formula-2}, we
have

\begin{equation}\begin{aligned}\label{gap1}\left(\lambda_{k+2}-\lambda_{k+1}\right)^{2}
\leq\frac{32\overline{\alpha}^{2}(\lambda_{k+2}+c)}{n\alpha^{2}+(N)\beta}\left(\lambda_{1}+c\right)
.
\end{aligned}\end{equation}
Furthermore, we deduce from \eqref{gap1} and \eqref{zlz-upper} that,
\begin{equation*}\begin{aligned}\lambda_{k+2}-\lambda_{k+1}&\leq
\sqrt{\frac{32\overline{\alpha}^{2}}{n\alpha^{2}+(n+p)\beta}}\sqrt{\lambda_{1}+c}\sqrt{\lambda_{k+2}+c}\\&\leq
(\lambda_{1}+c)\sqrt{\frac{32\overline{\alpha}^{2}C_{0}(n)}{n\alpha^{2}+(n+p)\beta}}(k+1)^{\frac{1}{n}}
\\&=C_{n,\Omega,f}(k+1)^{\frac{1}{n}},\end{aligned}
\end{equation*}
where $$C_{n,\Omega,f}=(\lambda_{1}+c)\sqrt{\frac{32\overline{\alpha}^{2}C_{0}(n)}{n\alpha^{2}+(n+p)\beta}}.$$
Therefore, we finish the proof of theorem \ref{thm1.1}.

$$\eqno\Box$$

According to the proof of theorem \ref{thm1.1}, it is not difficult to obtain the following corollary:

\begin{corr}\label{corr-3.2}
Let $(M^{n},g,f)$ be  an $n$-dimensional complete metric measure space isometrically
immersed in a Euclidean space $\mathbb{R}^{n+p}$, and $\lambda_{i}$ be the $i$-th $(i=1,2,\cdots,k)$
eigenvalue of the Dirichlet problem \eqref{Eigenvalue-Problem}. Then, for any
$k=1,2,\cdots,$ there are $(n+p+1)$ constants $\alpha,$ and $b_{j},j=1,2,\cdots,n+p$, such that

\begin{equation}\label{z-3}\begin{aligned}\lambda_{k+1}-\lambda_{k}\leq
C_{n,\Omega,f}k^{\frac{1}{n}},\end{aligned}
\end{equation}
where $$C_{n,\Omega,f}=(\lambda_{1}+c)\sqrt{\frac{32\overline{\alpha}^{2}C_{0}(n)}{n\alpha^{2}+\sum_{j=1}^{n+p}b_{j}}},$$$$c=\frac{1}{4}\int_{\Omega}u^{2}_{i}\left(|\nabla f|^{2}+2|\Delta_{f}f|
+n^{2}H^{2}\right)e^{-f}dv,$$and $C_{0}(n)$ is the same as the one in \eqref{Cheng-Yang-ineq}.
\end{corr}

\begin{rem}If $M^{n}$ is an $n$-dimensional Euclidean space, then we have $H=0$, and thus $c=0$. Let $\alpha_{j}=1$, where $j=1,2,\cdots,n+p$, then $h_{j}=x^{i}$. Thus, we have

$$\alpha=1,$$ and

$$\sum_{j=1}^{n+p}b_{j}=n,$$which implies that

$$C_{n,\Omega,f}=(\lambda_{1}+c)\sqrt{\frac{32\overline{\alpha}^{2}C_{0}(n)}{n\alpha^{2}+\sum_{j=1}^{n+p}b_{j}}}
=4\lambda_{1}\sqrt{\frac{C_{0}(n)}{n}}.$$ By \eqref{lemma-3.1-4}, we know that the assumption \eqref{sum-0} holds for the above case. Therefore, our result is sharper than Chen-Zheng-Yang's eigenvalue inequality \eqref{z-3}.
\end{rem}

\begin{rem}In corollary \ref{corr-3.2}, we assume that the complete Riemannian manifold $M^{n}$ is a minimal submanifold of the Euclidean space $\mathbb{R}^{n+p}$, then the constant $c$
is given by $$c=\frac{1}{4}\int_{M^{n}}u^{2}_{i}\left(|\nabla f|^{2}+2|\Delta_{f}f|\right)e^{-f}dv.$$
Furthermore, if $f$ is a constant, it is clear that the constant $c=0$.

\end{rem}
\textbf{\emph{Proof of theorem}} \ref{thm1.2}. By proposition \ref{general-formula-2} and lemma \ref{lem3.1}, we can give the proof by using the same method as the proof
of theorem \ref{thm1.1}.
$$\eqno\Box$$
Similarly, we have the following:
\begin{corr}\label{corr-3.3}
Let $(M^{n},g,f)$ be an $n$-dimensional closed metric measure space, and $\overline{\lambda}_{i}$ be the $i$-th $(i=0,1,2,\cdots,k)$
eigenvalue of the closed eigenvalue problem \eqref{Eigen-Prob-closed}. Then, for any
$k=1,2,\cdots,$ there are $(n+p+1)$ constants $\alpha,$ and $b_{j},j=1,2,\cdots,n+p$, such that

\begin{equation*}\begin{aligned}\overline{\lambda}_{k+1}-\overline{\lambda}_{k}\leq
C_{n,M^{n},f}k^{\frac{1}{n}},\end{aligned}
\end{equation*}
where $$C_{n,M^{n},f}=(\overline{\lambda}_{1}+c)\sqrt{\frac{32\overline{\alpha}^{2}C_{0}(n)}{n\alpha^{2}+\sum_{j=1}^{n+p}b_{j}}},$$and $C_{0}(n)$
is the same as the one in \eqref{Cheng-Yang-ineq}, and $$c=\frac{1}{4}\int_{M^{n}}u^{2}_{i}\left(|\nabla f|^{2}+2|\Delta_{f}f|
+n^{2}H^{2}\right)e^{-f}dv.$$
\end{corr}

\begin{rem}In corollary \ref{corr-3.3}, we assume that $M^{n}$ is a minimal submanifold of the Euclidean space $\mathbb{R}^{n+p}$, then the constant $c$
is given by $$c=\frac{1}{4}\int_{M^{n}}u^{2}_{i}\left(|\nabla f|^{2}+2|\Delta_{f}f|\right)e^{-f}dv.$$
Furthermore, if $f$ is a constant, then we know that the constant $c=0$.

\end{rem}

In theorem \ref{thm1.1}, the best constant $C_{n,\Omega,f}$ is called the gap coefficient. We shall note that it is worth noting
that it is very difficult for us to give the explicit expression of the optimal gap coefficient, even if $\Omega$ are some special domains in
the Euclidean space with  dimension $n$ and $f$ is a constant. Let $\Omega$ be a bounded domain with piecewise smooth boundary $\partial\Omega$ on an $n$-dimensional Riemannian manifold
$M^{n}$. If $\lambda_{i}$ is the $i$-th eigenvalue of Dirichlet problem \eqref{Eigen-Prob-Lapl}. According to a great amount of  numeric calculation for some special examples,
the author conjectured that \cite{Z3}: for any positive integer $k$,

\begin{equation*} \lambda_{k+1}-\lambda_{k}\leq (\lambda_{2}-\lambda_{1})k^{\frac{1}{n}}.
\end{equation*}
Therefore, in the sense of metric measure space, it is natural to generalize the above conjecture to the following:
\begin{con}\label{gap-conj-z}
Let $\Omega$ be a bounded domain with piecewise smooth boundary $\partial\Omega$ on an $n$-dimensional Riemannian manifold
$M^{n}$. If $\lambda_{i}$ is the $i$-th eigenvalue of Dirichlet problem \eqref{Eigenvalue-Problem} of the drifting Laplacian. Then, for any positive integer $k$,

\begin{equation}\label{con-z-1} \lambda_{k+1}-\lambda_{k}\leq (\lambda_{2}-\lambda_{1})k^{\frac{1}{n}}.
\end{equation}
\end{con}

\begin{rem} As we know, for the Dirichlet problem \eqref{Eigenvalue-Problem} on the Riemannian manifolds, the gap of the consecutive eigenvalues $\lambda_{k+1}-\lambda_{k}$ is
bounded by the first $k$-th eigenvalues in the previous literatures. However, from the above conjecture, we know that the gap of the consecutive eigenvalues
is bounded only by the first two eigenvalues.  \end{rem}

\section{Eigenvalues on the Ricci Solitons}\label{sec4}
\vskip3mm

 As an application of general formula of eigenvalues of
drifting Laplacian on complete metric measure spaces, we will consider
the gradient Ricci solitons in this section. Recall that Ricci solitons play an important role as
singularity dilations in the Ricci flow proof of uniformization, see
\cite{CK}. They correspond to self-similar solutions of Ricci flow
\cite{Ha1}, and usually serve as natural generalizations of
Einstein metrics. Assume that $S$ denotes the scalar curvature of $M^{n}$, then we have the following

\begin{prop}
\label{prop4.1} For  an $n$-dimensional closed gradient Ricci soliton $(M^{n},g,f)$, for any $k$,
 eigenvalues  of the closed eigenvalue problem \eqref{Eigen-Prob-closed}
of drifting Laplacian satisfy
\begin{equation}\begin{aligned}\label{zlz-improve-ineq}\sum^{k}_{i=0}\left(\overline{\lambda}_{k+1}-\overline{\lambda}_{i}\right)^{2}\leq \sum^{k}_{i=0}
\left(\overline{\lambda}_{k+1}-\overline{\lambda}_{i}\right)\left(\overline{\lambda}_{i}+c\right),\end{aligned}
\end{equation}
where
   \begin{equation*}c=\frac{1}{4}\inf_{\psi\in\Psi}\max_{M^{n}}\left(n^{2}H^{2}+4|\rho f-\rho\overline{c}|
+2\rho f+n\rho-2\rho\overline{c}-S\right),\end{equation*} and
\begin{equation*}\overline{c}=\frac{\int_{M^{n}}fe^{-f}dv}{\int_{M^{n}}e^{-f}dv}.\end{equation*}
\end{prop}
\begin{proof}By making use of equation  \eqref{soliton-equation}, we have (cf. \cite{CX1, FS}):

\begin{equation}\label{4.1}
S+\Delta f=n\rho,
\end{equation}
and
\begin{equation}
\label{4.2}S+|\nabla f|^{2}=2\rho f+\widetilde{c},
\end{equation}
where $S$ denotes the scalar curvature of $M^{n}$ and
$\widetilde{c}$ is a constant. From (\ref{4.1}) and
(\ref{4.2}), we have

\begin{equation}\label{4.3}\Delta_{f}f=n\rho-2\rho f-\widetilde{c}.\end{equation}
Therefore, by integrating for (\ref{4.3}), we obtain

\begin{equation}\label{4.4}\widetilde{c}=n\rho-2\rho\frac{\int_{M^{n}}fe^{-f}dv}{\int_{M^{n}}e^{-f}dv}.\end{equation}
By making use of (\ref{4.1}), (\ref{4.2}) and (\ref{4.4}), we
have

\begin{equation}\label{4.5}2|\Delta _{f}f|+|\nabla f|^{2}
=|2n\rho-4\rho f-2\widetilde{c}|+2\rho f+\widetilde{c}-S.\end{equation}Hence, from (\ref{4.5}), we obtain

\begin{equation}\label{z-4.7}\begin{aligned}&\int_{M^{n}}u_{i}^{2}(2|\Delta_{f} f|+|\nabla
f|^{2})e^{-f}dv
\\&=\int_{M^{n}}u_{i}^{2}\left(|2n\rho-4\rho f-2\widetilde{c}|+2\rho f+\widetilde{c}-S\right)e^{-f}dv
\\&=\int_{M^{n}}u_{i}^{2}\left(4|\rho f-\rho\overline{c}|
+2\rho f+n\rho-2\rho\overline{c}-S\right)e^{-f}dv,\end{aligned}
\end{equation}where $$\overline{c}=\frac{\int_{M^{n}}fe^{-f}dv}{\int_{M^{n}}e^{-f}dv}.$$
Recall that, in \cite{CZ}, Cheng and the author proved the following(also see \cite{Z}):
\begin{equation}\label{c-z-ineq}
\begin{aligned}
&\sum^{k}_{i=0}(\overline{\lambda}_{k+1}-\overline{\lambda}_{i})^{2}\\
&\leq\frac{4}{n}\sum^{k}_{i=0}(\overline{\lambda}_{k+1}-\overline{\lambda}_{i})
\left(\overline{\lambda}_{i}+\frac{1}{4}\int_{M^{n}}u_{i}^{2}(n^{2}H^{2}+2\Delta
f-|\nabla f|^{2})e^{-f}dv\right).
\end{aligned}
\end{equation}Therefore, it follows from \eqref{c-z-ineq} that ,\begin{equation}\label{z-ineq}
\begin{aligned}
&\sum^{k}_{i=0}(\overline{\lambda}_{k+1}-\overline{\lambda}_{i})^{2}\\
&\leq\frac{4}{n}\sum^{k}_{i=0}(\overline{\lambda}_{k+1}-\overline{\lambda}_{i})
\left(\overline{\lambda}_{i}+\frac{1}{4}\int_{M^{n}}u_{i}^{2}(n^{2}H^{2}+2|\Delta_{f}
f|+|\nabla f|^{2})e^{-f}dv\right)
\\
&\leq\frac{4}{n}\sum^{k}_{i=0}(\overline{\lambda}_{k+1}-\overline{\lambda}_{i})
\left(\overline{\lambda}_{i}+c\right),
\end{aligned}
\end{equation}where$$c=\frac{1}{4}\inf_{\psi\in\Psi}\max_{M^{n}}\left(n^{2}H^{2}+4|\rho f-\rho\overline{c}|
+2\rho f+n\rho-2\rho\overline{c}-S\right).$$ Therefore, we finish the proof of this proposition.

\end{proof}

\begin{prop}
\label{prop4.1} Assume that  $M^{n}$ is a submanifold in the Euclidean space $\mathbb{R}^{n+p}$ and
$H$ is the mean curvature of the submanifold $M^{n}$. For  an $n$-dimensional complete gradient Ricci soliton
$(M^{n},g,f)$, there exists a function $H$ such that, for any $k$,
 eigenvalues  of the Dirichlet problem \eqref{Eigenvalue-Problem}
of drifting Laplacian satisfy
\begin{equation}\begin{aligned}\label{zlz-improve-ineq}\sum^{k}_{i=1}\left(\lambda_{k+1}-\lambda_{i}\right)^{2}\leq \sum^{k}_{i=1}
\left(\lambda_{k+1}-\lambda_{i}\right)\left(\lambda_{i}+c\right),\end{aligned}
\end{equation}
where
   \begin{equation*}c=\frac{1}{4}\inf_{\psi\in\Psi}\max_{\Omega}\left(n^{2}H^{2}+4|\rho f-\rho\overline{c}|
+2\rho f+n\rho-2\rho\overline{c}-S\right),\end{equation*} and
\begin{equation*}\overline{c}=\frac{\int_{\Omega}fe^{-f}dv}{\int_{\Omega}e^{-f}dv}.\end{equation*}
\end{prop}

\begin{lem}\label{lem4.2}
For an $n$-dimensional closed Ricci soliton  $(M^{n},g,f)$, the
$k^{\text{th}}$ eigenvalue $\lambda_{k}$ of the eigenvalue
problem {\rm \eqref{Eigen-Prob-closed}} of  the drifting Laplacian satisfy, for any
$k\geq 1$,
\begin{equation*}
\overline{\lambda}_{k+1}+ c
\leq (1+\frac {4}{n})(\overline{\lambda}_{1}+c) \ k^{2/n},
\end{equation*}
where
 $$c=\frac{1}{4}\inf_{\psi\in\Psi}\max_{M^{n}}\left(n^{2}H^{2}+4|\rho f-\rho\overline{c}|
+2\rho f+n\rho-2\rho\overline{c}-S\right),$$ and
\begin{equation*}\overline{c}=\frac{\int_{M^{n}}fe^{-f}dv}{\int_{M^{n}}e^{-f}dv}.\end{equation*}
\end{lem}
\begin{proof} Putting
$$
\mu_{i+1}=\lambda_{i}+c>0,
$$
for any $i=0, 1, 2, \cdots$.  Then, we obtain from \eqref{zlz-improve-ineq}
\begin{equation}
\sum_{i=1}^k (\mu_{k+1}-\mu_{i})^2 \leq \frac{4}{n}\sum_{i=1}^k
(\mu_{k+1}-\mu_{i})\mu_i.
\end{equation}
By making use of the same proof as in Cheng and Yang \cite{CY2}, we can complete
our proof of lemma \ref{lem4.2}.
\end{proof}
Applying proposition   \ref{prop4.1} and lemma \ref{lem4.2}, we shall give the proof of theorem \ref{thm-z-3}.
\vskip 3mm
\textbf{\emph{Proof of theorem}} \ref{thm-z-3}.
Let $F_{j}(x)=\alpha_{j}x^{j}$ and $a_{j}>0$, such that

\begin{equation*}a_{j}^{2}=\|\nabla F_{j}u_{i}\|_{M^{n}}^{2}=\sqrt{\||\nabla F_{j}|^{2}u_{i}\|_{M^{n}}^{2}}=b_{j}\geq0,\end{equation*}

\begin{equation*}
\begin{aligned} \sum_{j=1}^{n+p}\int2 u_{i}\langle\nabla F_{j},\nabla f\rangle\Delta  F_{j}e^{-f}dv=0,\end{aligned}\end{equation*}
and \begin{equation*}\begin{aligned}  \sum_{j=1}^{n+p}\int2 u_{i}\langle\nabla F_{j},\nabla u_{i}\rangle\Delta  F_{j}e^{-f}dv=0,\end{aligned}\end{equation*} where
$j=1,2,\cdots,n+p,$ and $x^{j}$ denotes the $j$-th standard coordinate function of the Euclidean space $\mathbb{R}^{N}$.
Let $$\alpha=\min_{1\leq j\leq n+p}\{\alpha_{j}\},$$$$\overline{\alpha}=\max_{1\leq j\leq N}\{\alpha_{j}\},$$$$\beta=\min_{1\leq j\leq N}\{b_{j}\}.$$
Then, we have

\begin{equation}\begin{aligned}\label{4.12}\sum_{j=1}^{l}\frac{a_{j}^{2}+b_{j}}{2}\geq\frac{1}{2}\left(n\alpha^{2}+(n+p)\beta\right).\end{aligned}\end{equation}
By the same argument as the proof of theorem \ref{thm1.1},  we deduce

\begin{equation}\begin{aligned}\label{4.13}&4(\overline{\lambda}_{k+2}+c)\sum_{j=1}^{n+p}\|2\langle\nabla F_{j},\nabla u_{i}\rangle+u_{i}\Delta_{f} F_{j}\|_{M^{n}}^{2}
\\&\leq4(\overline{\lambda}_{k+2}+c)\overline{\alpha}^{2}\left(4\overline{\lambda}_{i}+\int_{M^{n}}u^{2}_{i}\left(|\nabla f|^{2}+2|\Delta_{f}f|
+n^{2}H^{2}\right)e^{-f}dv\right)
\\&\leq16(\overline{\lambda}_{k+2}+c)\overline{\alpha}^{2}\left(\overline{\lambda}_{i}+\frac{1}{4}\left(\int_{M^{n}}u^{2}_{i}\left(|\nabla f|^{2}+2|\Delta_{f}f|
\right)e^{-f}dv+\int_{M^{n}}u^{2}_{i}n^{2}H^{2}e^{-f}dv\right)\right).\end{aligned}\end{equation}
From the proof of proposition \ref{prop4.1} and inequality \eqref{4.13}, we infer that

\begin{equation}\begin{aligned}\label{4.14}&4(\overline{\lambda}_{k+2}+c)\sum_{j=1}^{N}\|2\langle\nabla F_{j},\nabla u_{i}\rangle+u_{i}\Delta_{f} F_{j}\|_{M^{n}}^{2}
\leq16\lambda_{k+2}\overline{\alpha}^{2}\left(\overline{\lambda}_{i}+c\right),\end{aligned}\end{equation}where
  $$c=\frac{1}{4}\inf_{\psi\in\Psi}\max_{M^{n}}\left(n^{2}H^{2}+4|\rho f-\rho\overline{c}|
+2\rho f+n\rho-2\rho\overline{c}-S\right),$$ and
\begin{equation*}\overline{c}=\frac{\int_{M^{n}}fe^{-f}dv}{\int_{M^{n}}e^{-f}dv}.\end{equation*}
Let $\tau=c$ in proposition \ref{gen-for-2}. Then, substituting \eqref{4.12} and \eqref{4.14} into \eqref{general-formula-4}, we obtain

\begin{equation*}\begin{aligned}\frac{1}{2}\left(n\alpha^{2}+(n+p)\beta\right)\left(\overline{\lambda}_{k+2}-\overline{\lambda}_{k+1}\right)^{2}
\leq16\lambda_{k+2}\overline{\alpha}^{2}\left(\overline{\lambda}_{i}+c\right),
\end{aligned}\end{equation*}
which implies that
\begin{equation*}\begin{aligned}\overline{\lambda}_{k+2}-\overline{\lambda}_{k+1}&\leq
\sqrt{\frac{32\overline{\alpha}^{2}}{n\alpha^{2}+(n+p)\beta}}\sqrt{\lambda_{1}+c}\sqrt{\lambda_{k+2}+c}\\&\leq
(\overline{\lambda}_{1}+c)\sqrt{\frac{32\overline{\alpha}^{2}C_{0}(n)}{n\alpha^{2}+(n+p)\beta}}(k+1)^{\frac{1}{n}}
\\&=C_{n,M^{n},f}(k+1)^{\frac{1}{n}},\end{aligned}
\end{equation*}
where $$C_{n,M^{n},f}=(\overline{\lambda}_{1}+c)\sqrt{\frac{32\overline{\alpha}^{2}C_{0}(n)}{n\alpha^{2}+(n+p)\beta}},$$ and $C_{0}(n)$ is the same as the one in \eqref{Cheng-Yang-ineq}.
Therefore, we finish the proof of this
theorem.

$$\eqno\Box$$

\begin{rem} For a complete gradient Ricci soliton $M^{n},g,f$, if it is a
minimal submanifold of $\mathbb{R}^{n+p}$, the constant $c$ in the
theorem {\rm\ref{thm-z-3}} will be given by
$$c=\frac{1}{4}\left(n\rho+2\rho\overline{c}-\max_{M^{n}}(2\rho f+R)\right),$$
and
$$\overline{c}=n\rho-2\rho\frac{\int_{\Omega}fe^{-f}dv}{\int_{\Omega}e^{-f}dv}.$$\end{rem}

\begin{rem} The constant $c$, which is appeared in \cite{CZ}, is given by
\begin{equation*}c=\frac{1}{4}\left(n\rho+2\rho\overline{c}
+\inf_{\psi\in\Psi}\max_{M^{n}}(n^{2}H^{2}-2\rho
f-R)\right),\end{equation*} and
\begin{equation*}\overline{c}=\frac{\int_{M^{n}}fe^{-f}dv}{\int_{M^{n}}e^{-f}dv}.\end{equation*} This is because it is not necessary for us to compute the value of $2\Delta f-|\nabla f|^{2}$
but $2|\Delta_{f}f|+|\nabla f|^{2}|$ in proposition \ref{prop4.1}. \end{rem}

\begin{rem}\label{rem4.1}
Assume that $(M^{n},g_{ij},f)$ is a compact, expanding or steady, gradient Ricci soliton, then, the
gradient Ricci solitons is Einstein \cite{Ha3,Iv}, which means that $$2|\Delta_{f}f|+|\nabla f|^{2}=0.$$ Hence, the constant $c$ in the
theorem {\rm \ref{thm5.5}} can be given by
\begin{equation*} c=\frac{1}{4}\inf_{\psi\in \Psi} n^{2}H^{2}.
\end{equation*}
\end{rem}

\begin{rem}\label{rem4.2} We suppose $(M, g)$ is a Sasakian manifold satisfying the gradient Ricci soliton
equation, and then $f$ is a constant function. Therefore, $(M, g)$ is an Einstein manifold \cite{HZ}, which implies
that there dose not exist the compact non-Einstein Ricci soliton in Sasakian manifolds since all compact Ricci solitons are gradient ones from \cite{Pere1}.
For this case, the constant $c$ in the
theorem {\rm \ref{thm5.5}} can be given by
\begin{equation*} c=\frac{1}{4}\inf_{\psi\in \Psi} n^{2}H^{2}.
\end{equation*}
\end{rem}

\begin{rem}\label{rem4.3}
For  a compact shrinking Ricci soliton $(M^{n},g_{ij},f)$ with dimension $n\leq3$, the
gradient Ricci Solitons is Einstein \cite{Ha3,Iv}. Hence, the constant $c$ in the
theorem {\rm \ref{thm5.5}} can be given by
\begin{equation*} c=\frac{1}{4}\inf_{\psi\in \Psi} n^{2}H^{2}.
\end{equation*}
\end{rem}

\begin{rem} If $(M^{n},g,f)$ is a compact Ricci soliton with rigidity, then the constant $c$ in the
theorem {\rm \ref{thm5.5}} can be given by
\begin{equation*} c=\frac{1}{4}\inf_{\psi\in \Psi} n^{2}H^{2}.
\end{equation*}Indeed, since $(M^{n},g,f)$ is a compact Ricci soliton with rigidity, then it is a trivial Ricci soliton which means that $f$ is a constant \cite{ME2}. Therefore, we have
$$2|\Delta_{f}f|+|\nabla f|^{2}=0.$$\end{rem}

\begin{thm}
\label{thm-soliton} Let $(M^{n},g_{ij},f)$ be an $n$-dimensional complete
gradient Ricci Soliton. Then, for any $k$, eigenvalues of the Dirichlet
problem \eqref{Eigenvalue-Problem} of the drifting Laplacian satisfy

\begin{equation}\label{z-3}\begin{aligned}\lambda_{k+1}-\lambda_{k}\leq
C_{n,\Omega,f}k^{\frac{1}{n}},\end{aligned}
\end{equation}
where \begin{equation}\label{c-1}C_{n,\Omega,f}=(\lambda_{1}+c)\sqrt{\frac{32\overline{\alpha}^{2}C_{0}(n)}{n\alpha^{2}+\sum_{j=1}^{n+p}b_{j}}},\end{equation} $C_{0}(n)$ is the same as the one in \eqref{Cheng-Yang-ineq},
  $$c=\frac{1}{4}\inf_{\psi\in\Psi}\max_{\Omega}\left(n^{2}H^{2}+4|\rho f-\rho\overline{c}|
+2\rho f+n\rho-2\rho\overline{c}-S\right),$$ and
\begin{equation*}\overline{c}=\frac{\int_{\Omega}fe^{-f}dv}{\int_{\Omega}e^{-f}dv}.\end{equation*}
\end{thm}
\begin{proof}The proof is similar to the theorem \ref{thm-z-3}. Thus, we omit it here.\end{proof}

\begin{rem} In theorem \ref{thm-soliton}, we further assume that  $(M^{n},g)$ is an $n$-dimensional complete minimal
submanifold of the $(n+p)$-dimensional Euclidean space $\mathbb{R}^{n+p}$. Then, the mean curvature is
zero and thus it is not difficult to see that the constant $c$ in theorem \ref{thm-soliton} will be
given by
 \begin{equation*}c=\frac{1}{4}\inf_{\psi\in\Psi}\max_{\Omega}\left(4|\rho f-\rho\overline{c}|
+2\rho f+n\rho-2\rho\overline{c}-S\right),\end{equation*}where $$\overline{c}=\frac{\int_{\Omega}fe^{-f}dv}{\int_{\Omega}e^{-f}dv}.$$\end{rem}

\begin{rem}If we assume that $(M^{n},g,f)$ is a steady Ricci soliton, then the constant $c$ in theorem \ref{thm-soliton} is given by
\begin{equation*}c=\frac{1}{4}\inf_{\psi\in\Psi}\max_{M^{n}}\left(n^{2}H^{2}-S\right).\end{equation*}\end{rem}

\begin{thm}{\rm\cite{ME}}Let $(M^{n}, g)$ be an $n$-dimensional complete noncompact gradient shrinking
Ricci soliton whose curvature tensor has at most exponential growth and having Ricci tensor
bounded from below.  Then, for any $k,k=1,2,\cdots$, eigenvalues of the Dirichlet
problem \eqref{Eigenvalue-Problem} of the drifting Laplacian satisfy

\begin{equation}\label{z-3}\begin{aligned}\lambda_{k+1}-\lambda_{k}\leq
C_{n,\Omega,f}k^{\frac{1}{n}},\end{aligned}
\end{equation}
where \begin{equation}\label{c-1}C_{n,\Omega,f}=(\lambda_{1}+c)\sqrt{\frac{32\overline{\alpha}^{2}C_{0}(n)}{n\alpha^{2}+\sum_{j=1}^{n+p}b_{j}}},\end{equation} $C_{0}(n)$
is the same as the one in \eqref{Cheng-Yang-ineq},
   \begin{equation*}c=\frac{1}{4}\inf_{\psi\in\Psi}\max_{\Omega}\left(4|\rho f-\rho\overline{c}|
+2\rho f+n\rho-2\rho\overline{c}-S\right),\end{equation*}where $$\overline{c}=\frac{\int_{\Omega}fe^{-f}dv}{\int_{\Omega}e^{-f}dv}.$$
\end{thm}

Let $(M^{n},g)$ be an $n$-dimensional, complete manifold. Suppose that there exists a smooth function
$f : M\rightarrow \mathbb{R}$ satisfying ${\rm Hess} f=\rho g$, for some constant $\rho\neq0$. Then Riemannian manifold $M^{n}$ is isometric to
$\mathbb{R}^{n}$. Hence, we have

\begin{corr}Let $(M^{n},g,f)$ be an $n$-dimensional, complete gradient Ricci soliton with ${\rm Hess}f=\rho g$.
Assume that $\lambda_{i}$ denotes the $i$-th eigenvalue of Dirichlet problem Dirichlet problem \eqref{Eigenvalue-Problem} of the drifting Laplacian.
Then, for any $k=1,2,\cdots,$ we have

\begin{equation}\label{z-3}\begin{aligned}\lambda_{k+1}-\lambda_{k}\leq
C_{n,\Omega,f}k^{\frac{1}{n}},\end{aligned}
\end{equation}
where \begin{equation}\label{c-1}C_{n,\Omega,f}=(\lambda_{1}+c)\sqrt{\frac{32\overline{\alpha}^{2}C_{0}(n)}{n\alpha^{2}+\sum_{j=1}^{n+p}b_{j}}},\end{equation} $C_{0}(n)$ is the same as the one in \eqref{Cheng-Yang-ineq},where \begin{equation*}c=\frac{1}{4}\left(n\rho+2\rho\overline{c}
-2\min_{\Omega}\rho
f\right),\end{equation*} and
\begin{equation*}\overline{c}=\frac{\int_{\Omega}fe^{-f}dv}{\int_{\Omega}e^{-f}dv}.\end{equation*}

\end{corr}

\begin{proof}If ${\rm Hess}f =\rho g$, then we have  \begin{equation}
\label{4.22}\Delta f=n\rho,\end{equation} and

\begin{equation}
\label{4.23}|\nabla f|^{2}=2\rho f+\widetilde{c},
\end{equation}
where
$\widetilde{c}$ is a constant defined by

\begin{equation*}\widetilde{c}=n\rho-2\rho\frac{\int_{\Omega}fe^{-f}dv}{\int_{\Omega}e^{-f}dv}.\end{equation*}
It follows from \eqref{4.22} and \eqref{4.23} that
\begin{equation}\label{4.24}\begin{aligned}&\int_{\Omega}u_{i}^{2}(2|\Delta_{f} f|+|\nabla
f|^{2})e^{-f}dv\\&=2\int_{\Omega}u_{i}^{2}\left|2\rho\frac{\int_{\Omega}fe^{-f}dv}{\int_{\Omega}e^{-f}dv}-2\rho f\right|e^{-f}dv
+2\rho \int_{\Omega}u_{i}^{2}fe^{-f}dv+n\rho-2\rho\frac{\int_{\Omega}fe^{-f}dv}{\int_{\Omega}e^{-f}dv}
\\&=4\int_{\Omega}u_{i}^{2}\left|\rho\overline{c}-\rho f\right|e^{-f}dv
+2\rho \int_{\Omega}u_{i}^{2}fe^{-f}dv+n\rho-2\rho\overline{c},\end{aligned}
\end{equation}where $$\overline{c}=\frac{\int_{\Omega}fe^{-f}dv}{\int_{\Omega}e^{-f}dv}.$$
Hence, by the same method as the proof of proposition \ref{prop4.1}, we have

\begin{equation*}
\begin{aligned}
\sum^{k}_{i=1}(\lambda_{k+1}-\lambda_{i})^{2}
\leq\frac{4}{n}\sum^{k}_{i=1}(\lambda_{k+1}-\lambda_{i})
\left(\lambda_{i}+c\right),
\end{aligned}
\end{equation*}where \begin{equation*}c=\frac{1}{4}\min_{\psi\in\Psi}\max_{\Omega}\left\{n^{2}H^{2}+4\int_{\Omega}u_{i}^{2}\left|\rho\overline{c}-\rho f\right|e^{-f}dv
+2\rho \int_{\Omega}u_{i}^{2}fe^{-f}dv+n\rho-2\rho\overline{c}\right\},\end{equation*}and  $$\overline{c}=\frac{\int_{\Omega}fe^{-f}dv}{\int_{\Omega}e^{-f}dv}.$$
By the recursion formula of Cheng and Yang, we get
\begin{equation*}
\lambda_{k+1}+ c
\leq (1+\frac {4}{n})(\lambda_{1}+c) \ k^{2/n}.
\end{equation*}
By the argument in \cite{PRS}, we know that $M^{n}$ is isometric to
$\mathbb{R}^{n}$. Therefore, by the same method as in the proof of theorem \ref{thm-z-3}, we can get
\begin{equation*}\begin{aligned}\lambda_{k+1}-\lambda_{k}\leq
C_{n,\Omega,f}k^{\frac{1}{n}},\end{aligned}
\end{equation*}
where \begin{equation*}C_{n,\Omega,f}=(\lambda_{1}+c)\sqrt{\frac{32\overline{\alpha}^{2}C_{0}(n)}{n\alpha^{2}+\sum_{j=1}^{n+p}b_{j}}},\end{equation*} $C_{0}(n)$
is the same as the one in \eqref{Cheng-Yang-ineq}.
This finishes the proof of this corollary.
\end{proof}

\begin{corr}Let $(M^{n},g,f)$ be a complete, expanding, Ricci soliton.  If the scalar
curvature $S\geq 0$ and $S\in L^{1}(M^{n},e^{-f}dv)$. Assume that $\lambda_{i}$ denotes the $i$-th eigenvalue of Dirichlet problem Dirichlet
problem \eqref{Eigenvalue-Problem} of the drifting Laplacian.
Then, for any $k=1,2,\cdots,$ one has

\begin{equation}\label{z-31}\begin{aligned}\lambda_{k+1}-\lambda_{k}\leq
C_{n,\Omega,f}k^{\frac{1}{n}},\end{aligned}
\end{equation}
where \begin{equation}\label{c-1}C_{n,\Omega,f}=(\lambda_{1}+c)\sqrt{\frac{32\overline{\alpha}^{2}C_{0}(n)}{n\alpha^{2}+\sum_{j=1}^{n+p}b_{j}}},\end{equation} $C_{0}(n)$ is the same as the one in \eqref{Cheng-Yang-ineq},
where
\begin{equation*}c=\frac{1}{4}\left(n\rho+2\rho\overline{c}
-2\rho\max_{\Omega}
f\right),\end{equation*}
and
\begin{equation*}\overline{c}=\frac{\int_{\Omega}fe^{-f}dv}{\int_{\Omega}e^{-f}dv}.\end{equation*}

\end{corr}
\begin{proof}Since  the scalar curvature $S\geq 0$, we have

\begin{equation*}\begin{aligned}\frac{1}{4}\inf_{\psi\in\Psi}\max_{\Omega}&\left(n^{2}H^{2}+4|\rho f-\rho\overline{c}|
+2\rho f+n\rho-2\rho\overline{c}-S\right)\\&\leq\frac{1}{4}\inf_{\psi\in\Psi}\max_{\Omega}\left(n^{2}H^{2}+4|\rho f-\rho\overline{c}|
+2\rho f+n\rho-2\rho\overline{c}\right),\end{aligned}\end{equation*}
where
\begin{equation*}\overline{c}=\frac{\int_{\Omega}fe^{-f}dv}{\int_{\Omega}e^{-f}dv}.\end{equation*}
By the assumption of this corollary, we know that $M^{n}$ is isometric to the standard
Euclidean space \cite{PRS}. Since $\rho<0$ (i.e., $(M^{n},g,f)$ is an expanding Ricci soliton) and eigenvalues is invariant in the sense of isometries, we have the following eigenvalue inequality:

\begin{equation}\label{z-3}\begin{aligned}\lambda_{k+1}-\lambda_{k}\leq
C_{n,\Omega,f}k^{\frac{1}{n}},\end{aligned}
\end{equation}
where \begin{equation}\label{c-1}C_{n,\Omega,f}=(\lambda_{1}+c)\sqrt{\frac{32\overline{\alpha}^{2}C_{0}(n)}{n\alpha^{2}+\sum_{j=1}^{n+p}b_{j}}},\end{equation} $C_{0}(n)$ is the same as the one in \eqref{Cheng-Yang-ineq},where
\begin{equation*}c=-\frac{\rho}{4}\inf_{\psi\in\Psi}\max_{\Omega}\left(4|f-\overline{c}|
-2f-n+2\overline{c}\right),\end{equation*}where
\begin{equation*}\overline{c}=\frac{\int_{\Omega}fe^{-f}dv}{\int_{\Omega}e^{-f}dv}.\end{equation*}
Thus, it completes the proof of this corollary.\end{proof}
Let $S_{\ast}=\inf_{M^{n}}S.$  Assume that $(M^{n},g,f)$ is a geodesically complete expanding gradient Ricci soliton, then we have

\begin{corr}Let $(M^{n},g,f)$ be an dimensional, geodesically complete, expanding gradient Ricci soliton. If
$S_{\ast}\in(-\infty,n\rho)\cup(0,+\infty)$ or $S(x)\leq n\rho$. Assume that $\lambda_{i}$ denotes the $i$-th eigenvalue of
Dirichlet problem Dirichlet problem \eqref{Eigenvalue-Problem} of the drifting Laplacian.
Then, for any $k=1,2,\cdots,$ one has

\begin{equation}\label{z-4-3}\begin{aligned}\lambda_{k+1}-\lambda_{k}\leq
C_{n,\Omega}k^{\frac{1}{n}},\end{aligned}
\end{equation}
where \begin{equation*}\begin{aligned}c=\frac{1}{4}\inf_{\psi\in\Psi}\max_{\Omega}(n^{2}H^{2}),\end{aligned}\end{equation*}\begin{equation*}\label{c-1}C_{n,\Omega}=(\lambda_{1}+c)\sqrt{\frac{32\overline{\alpha}^{2}C_{0}(n)}{n\alpha^{2}+\sum_{j=1}^{n+p}b_{j}}},\end{equation*} $C_{0}(n)$ is the same as the one in \eqref{Cheng-Yang-ineq}.

\end{corr}

\begin{proof}Under the assumption of this corollary, and $S_{\ast}\in(-\infty,n\rho)\cup(0,+\infty)$ or $S(x)\leq n\rho$ is Einstein and the soliton is trivial.  Consequently, we have

\begin{equation}\label{f=const-1}|\nabla f|=0,\end{equation} and

\begin{equation}\label{f=const-2}\Delta f=0.\end{equation}
Furthermore, substituting \eqref{f=const-1} and \eqref{f=const-2} into \eqref{general-formula-2}, we obtain \eqref{z-4-3}.
This completes the proof of this corollary.
\end{proof}
Suppose that $(M^{n},g,f)$ is a complete shrinking Ricci soliton, then $S(x)>0$ on $M^{n}$ unless
$S(x) \equiv 0$ on $M^{n}$, and $M^{n}$ is isometric to $\mathbb{R}^{n}$ (see \cite{PRS}).  It is well known that eigenvalues is invariant in the sense of isometries,
therefore, we prove the following:

\begin{corr}Let $(M^{n},g,f)$ be an $n$-dimensional, complete shrinking Ricci soliton with scalar curvature function  $S(x)\leq0$ on $M^{n}$.
Assume that $\lambda_{i}$ denotes the $i$-th eigenvalue of Dirichlet problem Dirichlet problem \eqref{Eigenvalue-Problem} of the drifting Laplacian.
Then, for any $k=1,2,\cdots,$ one has

\begin{equation}\label{z-4-5}\begin{aligned}\lambda_{k+1}-\lambda_{k}\leq
C_{n,\Omega,f}k^{\frac{1}{n}},\end{aligned}
\end{equation}
where\begin{equation*}\label{c-1}C_{n,\Omega,f}=(\lambda_{1}+c)\sqrt{\frac{32\overline{\alpha}^{2}C_{0}(n)}{n\alpha^{2}+\sum_{j=1}^{n+p}b_{j}}},\end{equation*}
$C_{0}(n)$ is the same as the one in \eqref{Cheng-Yang-ineq}, and
\begin{equation*}c=\frac{1}{4}\max_{\Omega}\left(4|\rho f-\rho\overline{c}|
+2\rho f+n\rho-2\rho\overline{c}\right),\end{equation*}
and $$c=\frac{\int_{\Omega}fe^{-f}dv}{\int_{\Omega}e^{-f}dv}.$$
\end{corr}
Let $(M^{n},g,f)$ be complete,
 gradient (or expanding) Ricci soliton with $|\nabla f|\in
L^{p}(M^{n},e^{-f}dv)$, where $1\leq p\leq+\infty$, then, we have $\nabla f=0$ \cite{PRS}.   Then, one can easily prove the following:

\begin{corr}Let $(M^{n},g,f)$ be a complete gradient shrinking (or expanding) Ricci soliton with nonnegative Ricci curvature,
and contains a line. Assume that  where $|\nabla f|\in
L^{p}(M^{n},e^{-f}dv) $, and $1\leq p\leq+\infty$.
Assume that $\lambda_{i}$ denotes the $i$-th eigenvalue of Dirichlet problem Dirichlet problem \eqref{Eigenvalue-Problem} of the drifting Laplacian.
Then, for any $k=1,2,\cdots,$ one has

\begin{equation}\label{z-4-4}\begin{aligned}\lambda_{k+1}-\lambda_{k}\leq
C_{n,\Omega}k^{\frac{1}{n}},\end{aligned}
\end{equation}
where \begin{equation*}\begin{aligned}c=\frac{1}{4}\inf_{\psi\in\Psi}\max_{\Omega}(n^{2}H^{2}),\end{aligned}\end{equation*}\begin{equation*}\label{c-1}C_{n,\Omega}=(\lambda_{1}+c)\sqrt{\frac{32\overline{\alpha}^{2}C_{0}(n)}{n\alpha^{2}+\sum_{j=1}^{n+p}b_{j}}},\end{equation*} $C_{0}(n)$ is the same as the one in \eqref{Cheng-Yang-ineq}.

\end{corr}

The following theorem is to give an intrinsic eigenvalue inequality of drifting Laplcian on Ricci solitons.
\begin{thm}\label{thm2-z}Let $(M^{n},g,f)$ be a complete gradient shrinking Ricci soliton with nonnegative Ricci curvature,
and contains a line. Assume that $\lambda_{i}$ is the $i$-th
$(i=1,2,\cdots,k)$ eigenvalue of the Dirichlet problem \eqref{Eigenvalue-Problem}, then we have

\begin{equation*}\begin{aligned}\lambda_{k+1}-\lambda_{k}\leq4\sqrt{C_{0}(1)}\left(\lambda_{1}+c\right)k,
\end{aligned}\end{equation*}where $C_{0}(1)$ is the case that $n=1$ in \eqref{Cheng-Yang-ineq},
$\kappa=\frac{1}{2}Diameter(\Omega)+2\min_{x\in\Omega}\sqrt{f(x)}+\overline{c},$$$c=\frac{1}{4}\left(\frac{\rho}{2}\kappa^{2}+2\sqrt{2\rho}\kappa\lambda_{i}^{\frac{1}{2}}\right),$$
and
$$\overline{c}=n\rho-2\rho\frac{\int_{\Omega}fe^{-f}dv}{\int_{\Omega}e^{-f}dv}.$$
\end{thm}

\begin{proof}Assume that
$\gamma$ is a geodesic line on $(M^{n},g,f)$ and $B^{+}(x)$ is
Busemann function associated with $\gamma$.  and $h(x)=B^{+}(x)$. Then, we have $$|\nabla B^{+}(x)|^{2}=1$$ a.e. in $\Omega$. In \eqref{general-formula-2}, we suppose that $l=1$,
\begin{equation}\begin{aligned}\label{general-formula-5} \frac{a_{1}^{2}+b_{1}}{2}\left(\lambda_{k+2}-\lambda_{k+1}\right)^{2}
\leq4(\lambda_{k+2}+\rho)\|2\langle\nabla F_{1},\nabla u_{i}\rangle+u_{i}\Delta_{f}F_{1}\|^{2}
,
\end{aligned}\end{equation}Putting $F_{1}(x)=B^{+}(x)$
and substituting into the inequality \eqref{general-formula-5}, we have

\begin{equation*}a_{1}^{2}=\|\nabla F_{1}u_{i}\|_{\Omega}^{2}=\sqrt{\||\nabla F_{1}|^{2}u_{i}\|_{\Omega}^{2}}=b_{1}=1,\end{equation*}
which implies that
\begin{equation}\begin{aligned}\label{gen-for-6} \left(\lambda_{k+2}-\lambda_{k+1}\right)^{2}
\leq4(\lambda_{k+2}+\rho)\|u_{i}\Delta_{f}B^{+}+2\langle\nabla B^{+},\nabla u_{i}\rangle\|_{\Omega}^{2},
\end{aligned}\end{equation}Using equation \eqref{soliton-equation} and the contracted second Bianchi
identity, we have (see \cite{CH2} or Theorem 20.1 in
\cite{Ha3})

\begin{equation}\label{soliton-1}R+|\nabla f|^{2}-2\rho
f=\overline{c}\end{equation} for some constant $\overline{c}$. Here
$R$ denotes the scalar curvature of $g_{ij}$. By Lemma 2.3 in
\cite{CaZ}, we know that

\begin{equation}\label{c-z}|\nabla f|^{2}\leq\frac{\rho}{2}(r(x) +
2\sqrt{f(x_{0})}+\overline{c})^{2}.\end{equation} Here $r(x)=d(x_{0},x)$ is the distance
function  from some fixed point $x_{0}\in M^{n}$. Since $\Delta
B^{+}(x)=0$ (see \cite{SY}),
it  follows from \eqref{soliton-1} and \eqref{c-z} that (cf.\cite{Z1})

\begin{equation}\label{zl}\begin{aligned}\|u_{i}\Delta_{f}B^{+}+2\langle\nabla
B^{+},\nabla u_{i}\rangle\|_{\Omega}^{2}
\leq4\lambda_{i}+\frac{\rho}{2}\kappa^{2}+2\sqrt{2\rho}\kappa\lambda_{i}^{\frac{1}{2}},\end{aligned}\end{equation}where
$\kappa=\frac{1}{2}Diameter(\Omega)+2\min_{x\in\Omega}\sqrt{f(x)}+\overline{c},$
and $\overline{c}$ is a constant satisfies $R+|\nabla
f|^{2}-f=\overline{c}.$
Inserting \eqref{zl} into \eqref{gen-for-6}, we get

\begin{equation}\begin{aligned}\left(\lambda_{k+2}-\lambda_{k+1}\right)^{2}
\leq4(\lambda_{k+2}+\rho)\left(4\lambda_{i}+\frac{\rho}{2}\kappa^{2}+2\sqrt{2\rho}\kappa\lambda_{i}^{\frac{1}{2}}\right).
\end{aligned}\end{equation}Putting

$$c=\frac{1}{4}\left(\frac{\rho}{2}\kappa^{2}+2\sqrt{2\rho}\kappa\lambda_{i}^{\frac{1}{2}}\right),$$ then we have

\begin{equation}\begin{aligned}\label{gen-for-7} \lambda_{k+2}-\lambda_{k+1}
\leq4\sqrt{(\lambda_{k+2}+c)\left(\lambda_{i}+c\right)}.
\end{aligned}\end{equation}
In \cite{Z1}, the author proved the following eigenvalue inequality:
\begin{equation}\label{ineq-z-2}\begin{aligned}\sum^{k}_{i=1}(\lambda_{k+1}-\lambda_{i})^{2}
\leq\sum^{k}_{i=1}(\lambda_{k+1}-\lambda_{i})\left(4\lambda_{i}+\frac{\rho}{2}\kappa^{2}
+2\sqrt{2\rho}\kappa\lambda_{i}^{\frac{1}{2}}\right).\end{aligned}\end{equation}
Therefore, by Cheng and Yang's recursion formula, we have

\begin{equation}\begin{aligned}\label{gen-for-8}\lambda_{k+1}+c
\geq C_{0}(1) \left(\lambda_{1}+c\right) k^{2}.
\end{aligned}\end{equation}
Synthesizing \eqref{ineq-z-2} and \eqref{gen-for-8}, we yield

\begin{equation*}\begin{aligned}\lambda_{k+2}-\lambda_{k+1}&
\leq4\sqrt{(\lambda_{k+2}+c)\left(\lambda_{1}+c\right)}\\&\leq4\sqrt{C_{0}(1)}\left(\lambda_{1}+c\right)(k+1),
\end{aligned}\end{equation*}where $C_{0}(1)$ is the same as \eqref{Cheng-Yang-ineq}. Hence, we finish the proof of this theorem.
\end{proof}

\section{Applying to Some Important Solitons}\label{sec5}

It is well known that  nontrivial compact Ricci solitons
may exist only when the dimension of  $M^{n}$ is larger then 3. and they must have nonconstant positive scalar curvature
$(S> 0)$. However, complete noncompact examples exist in any dimension as the Gaussian soliton
(i.e., the radial vector field on the Euclidean space) shows. In this section, the first object we consider is
the Gaussian soliton, which is introduced by Hamilton in \cite{Ha2}.
It is not difficult to check that the flat Euclidean space
$(\mathbb{R}^{n}, \delta_{ij} )$ is a gradient shrinker with
potential function $f = |x|^{2}/4$:
$$\partial_{i}\partial_{j}f=\frac{1}{2}\delta_{ij}.$$ The Gaussian
shrinking soliton is exactly the triple $(\mathbb{R}^{n},\delta_{ij},|x|^{2}/4)$. We investigat the eigenvalue of drifting Laplacian on the Gaussian soliton and prove the following:

\begin{thm}
Let $(M^{n},g,f)$ be the Gaussian
shrinking soliton and $\lambda_{i}$ be the $i$-th $(i=1,2,\cdots,k)$
eigenvalue of the eigenvalue problem \eqref{Eigenvalue-Problem}. Then,

\begin{equation*}\begin{aligned}\lambda_{k+1}-\lambda_{k}\leq
C_{n,\Omega}k^{\frac{1}{n}},\end{aligned}
\end{equation*}
where $$C_{n,\Omega}=4(\lambda_{1}+c_{1})\sqrt{\frac{C_{0}(n)}{n}}.$$
\end{thm}
\begin{proof} Let $F_{j}(x)=x^{j},$ where $x^{\j}$ denotes the $j$-th local coordinate of $x_{0}\in\Omega\subset\mathbb{R}^{n}$.
Hence, we have $|\nabla x^{j}|=1$ and $\Delta x^{j}=0$. Let $l=n$, then, from \eqref{general-formula-2}, we
have

\begin{equation*}a_{j}=\sqrt{\|\nabla F_{j}u_{i}\|^{2}}=\sqrt{\|\nabla x_{j}u_{i}\|^{2}}=1,\end{equation*}

\begin{equation*}b_{j}=\sqrt{\||\nabla F_{j}|^{2}u_{i}\|^{2}}=\sqrt{\||\nabla x_{j}|^{2}u_{i}\|^{2}}=1.\end{equation*}
Therefore, we have

\begin{equation*}
a_{j}=b_{j},\end{equation*}
and

\begin{equation}\begin{aligned}\sum_{j=1}^{n}\left(\lambda_{k+2}-\lambda_{k+1}\right)^{2}
\leq4(\lambda_{k+2}+\rho)\sum_{j=1}^{n}\|2\langle\nabla x_{j},\nabla u_{i}\rangle+u_{i}\Delta_{f}x_{j}\|^{2}.
\end{aligned}\end{equation}
Since

\begin{equation}\begin{aligned}\label{ineq-sum}\sum_{j=1}^{n}\|2\langle\nabla x_{j},\nabla u_{i}\rangle+u_{i}\Delta_{f}x_{j}\|^{2}&=4\int_{\Omega}\langle\nabla x^{\alpha},\nabla u_{1}\rangle e^{-f}dv-\int_{\Omega}\langle\nabla f,\nabla x^{\alpha}\rangle^{2}u_{1}^{2}e^{-f}dv\\&\ \ \ \ \ \ \ +2\int_{\Omega}\langle\nabla\langle\nabla f,\nabla x^{\alpha}\rangle,\nabla x^{\alpha}\rangle u_{1}^{2}e^{-f}dv\Bigg),\end{aligned}\end{equation}and $$\sum_{\alpha=1}^{n}\langle\nabla x^{\alpha},\nabla x^{\alpha}\rangle=n,$$ from \eqref{ineq-sum}, we have

\begin{equation}\label{gap11}\begin{aligned}n\left(\lambda_{k+2}-\lambda_{k+1}\right)^{2}
&\leq16(\lambda_{k+2}+c)\sum^{n}_{\alpha=1}\Bigg(\int_{\Omega}\langle\nabla x^{\alpha},\nabla u_{1}\rangle e^{-f}dv-\frac{1}{4}\int_{\Omega}\langle\nabla f,\nabla x^{\alpha}\rangle^{2}u_{1}^{2}e^{-f}dv
\\&\ \ \ \ \ \ \ +\frac{1}{2}\int_{\Omega}\langle\nabla\langle\nabla f,\nabla x^{\alpha}\rangle,\nabla x^{\alpha}\rangle u_{1}^{2}e^{-f}dv\Bigg)\\
&=16(\lambda_{k+2}+c)\Bigg(\lambda_{1}-\frac{1}{4}\int_{\Omega}\langle\nabla \left(\frac{|x|^{2}}{4}\right),\nabla\left(\frac{|x|^{2}}{4}\right)\rangle u_{1}^{2}e^{-f}dv
\\&\ \ \ \ \ \ \ +\frac{1}{2}\sum^{n}_{\alpha=1}\int_{\Omega}\langle\nabla\langle\nabla\left(\frac{|x|^{2}}{4}\right),\nabla x^{\alpha}\rangle,\nabla x^{\alpha}\rangle u_{1}^{2}e^{-f}dv\Bigg)\\
&=16(\lambda_{k+2}+c)\left(\lambda_{1}-\frac{1}{16}\int_{\Omega}|x|^{2}u_{1}^{2}e^{-f}dv
+\frac{n}{4}\int_{\Omega}u_{1}^{2}e^{-f}dv\right)\\
&\leq(\lambda_{k+2}+c)\left(16\lambda_{1}+4n-\min_{\Omega}|x|^{2}
\right)\\
&=16(\lambda_{k+2}+c)\left(\lambda_{1}+c
\right)
.
\end{aligned}\end{equation}
Therefore, we deduce from \eqref{gap1} that,

\begin{equation*}\begin{aligned}\lambda_{k+2}-\lambda_{k+1}&\leq
4\sqrt{\frac{\lambda_{k+2}+c}{n}(\lambda_{1}+c)}\\&\leq
4(\lambda_{1}+c)\sqrt{\frac{C_{0}(n)}{n}}(k+1)^{\frac{1}{n}}
\\&=C_{n,\Omega}(k+1)^{\frac{1}{n}},\end{aligned}
\end{equation*}
where $$C_{n,\Omega}=4(\lambda_{1}+c)\sqrt{\frac{C_{0}(n)}{n}}.$$
This completes the proof of Theorem \ref{thm1.1}.

\end{proof}

We suppose that $(\mathbb{N}^{m},\langle,\rangle)$ is any
$m$-dimensional Einstein manifold with Ricci curvature
$Ric(\textbf{w})\neq0,~\textbf{w}\in\mathbb{N}^{m}$, and
$f(\textbf{v},\textbf{w}):\mathbb{R}^{n-m}\times
\mathbb{N}^{m}\rightarrow R$ is defined by (cf. \cite{PRS})
\begin{equation}\label{4.3.51}f(\textbf{v},\textbf{w})=\frac{Ric(\textbf{w})}{2}
|\textbf{v}|^{2}_{\mathbb{R}^{n-m}} + \langle \textbf{v},
\textbf{B}\rangle_{\mathbb{R}^{n-m}}+\textbf{C},\end{equation}with
$\textbf{C} \in\mathbb{R}$ and $\textbf{B}\in \mathbb{R}^{n-m}$,
where $|\cdot|_{\mathbb{R}^{n-m}}$ denotes the standard inner on the
$(n-m)$-dimensional Euclidean space $\mathbb{R}^{n-m}$. Then, it is well known that the Riemannian product manifold
$$(\mathbb{R}^{n-m}\times\mathbb{N}^{m},\langle,\rangle_{\mathbb{R}^{n-m}}+\langle,\rangle_{\mathbb{N}^{m}},f)$$ is a (noncompact)
shrinking soliton. In
particular, we consider the unit round cylinder
$\mathbb{S}^{m}(1)\times\mathbb{R}^{n-m}$ which is a noncompact
shrinking soliton, and assume that
$\textbf{B}=\textbf{0}\in\mathbb{R}^{n-m}$ and $\textbf{C}=0$, i.e.,
by substituting them into (\ref{4.3.51}), we have
\begin{equation}\label{4.3.52}f(\textbf{v},\textbf{w})=\frac{(m-1)|\textbf{v}|^{2}_{\mathbb{R}^{n-m}}}{2}
.\end{equation}
Under the above assumption, we can prove the following theorem.

\begin{thm}\label{thm-product-soliton}\ \ Let $$(\mathbb{R}^{n-m}\times\mathbb{N}^{m},\langle,\rangle_{\mathbb{R}^{n-m}}+\langle,\rangle_{\mathbb{N}^{m}},f)$$
 be an $n$-dimensional, gradient Ricci soliton, with \begin{equation*} f(\textbf{v},\textbf{w})=\frac{(m-1)|\textbf{v}|^{2}_{\mathbb{R}^{n-m}}}{2}
.\end{equation*}
Let  $\lambda_{i}$ be the $i$-th $(i=1,2,\cdots,k)$
eigenvalue of the eigenvalue problem \eqref{Eigenvalue-Problem}. Then,  we have

\begin{equation*}\begin{aligned}\lambda_{k+1}-\lambda_{k}\leq
C_{n,\Omega}k^{\frac{1}{n}},\end{aligned}
\end{equation*}
where $$C_{n,\Omega}=4(\lambda_{1}+c_{1})\sqrt{\frac{C_{0}(n)}{n}}.$$ \end{thm}
This proof is given in the appendix.

\begin{rem}In the above theorem, we consider the special case: $\mathbb{R}^{n-m}\times \mathbb{N}^{m}=\mathbb{R}^{n-m}\times \mathbb{S}^{m}$,
$\textbf{B}=\textbf{0}\in\mathbb{R}^{n-m}$ and $\textbf{C}=\textbf{0}$.
However, for the general case, we can obtain similar eigenvalue inequality of of Dirichlet problem \eqref{Eigenvalue-Problem} by
the same argument as in the proof of theorem \ref{thm-product-soliton}.
\end{rem}

\begin{rem} Suppose that the dimension $n\geq3$, and $(M^{n},g,f,)$ is a complete, rotationally
invariant shrinking soliton structure on a on a manifold $M^{n}$,
which is diffeomorphic to one of $\mathbb{S}^{n}$, $\mathbb{R}^{n}$, or
$\mathbb{R}\times\mathbb{S}^{n-1}$. Then, one has (cf. {\rm \cite{Kot}}) \\(1)if
$M^{n}\cong\mathbb{S}^{n}$, then $g$ is isometric to a round sphere
and $f\equiv {\rm const}$; \\(2) if $M^{n}\cong\mathbb{R}^{n}$ ,
then $g$ is flat; \\(3) if
$M^{n}\cong\mathbb{R}\times\mathbb{S}^{n-1}$, then $g$ is isometric
to the standard cylinder $dr^{2} +
\omega^{2}_{0}g_{\mathbb{S}^{n-1}}$ of radius $\omega_{0}=
\sqrt{(n-2 )/\rho}$ and $f=f(r)=(n-2)r^{2}/(2\omega^{2}_{0})+{\rm
linear}$. According to the above classification of solitons, it is not difficult to   obtain similar upper bound of the consecutive eigenvalues of drifting Laplacian on those complete,
rotationally invariant shrinking solitons since eigenvalues are invariant in the sense of  isometry.\end{rem}

\section{Eigenvalues of Drifting Laplacian on Self-shrinkers}\label{sec6}

In this section, we consider that $X : M^{n}\rightarrow \mathbb{R}^{n+p}$ is an $n$-dimensional submanifold in the $(n+p)$-dimensional
Euclidean space $\mathbb{R}^{n+1}$. Let $\{e_{1},e_{2},\cdots,e_{n}\}$ be a local orthonormal basis of $M^{n}$ with respect to
the induced metric, and $\{\theta_{1},\theta_{2},\cdots,\theta_{n}\}$ be their dual $1$-forms. Let $e_{n+1},e_{n+2},\cdots,e_{n+p}$ be the local
unit orthonormal normal vector fields. Furthermore, we make the
following convention on the range of indices:
$$1\leq i,j,k,\ldots\leq n;$$$$n+1\leq
\alpha,\beta,\gamma,\ldots\leq n+p.$$
By Cartan lemma, we
have

$$h_{ij}^{\alpha}=h_{ji}^{\alpha}~(\forall \alpha,\forall i,j),$$
where $h_{ij}^{\alpha}$ is the components of the second fundament form.
The second fundamental form $h$ of $M^{n}$, the mean curvature vector $H$ and the norm square of the second fundamental form $A$ are
defined, respectively, by

$$h=\sum^{n+p}_{\alpha=n+1}\sum^{n}_{i,j=1}h_{ij}^{\alpha}\omega_{i}\otimes\omega_{j}e_{\alpha},$$$$H=\frac{1}{n}\sum^{n+p}_{\alpha=n+1}\sum^{n}_{i=1}h_{ii}^{\alpha}e_{\alpha},$$
and $$|A|^{2}=\sum_{i,j,\alpha}(h_{ij}^{\alpha})^{2}$$be the norm square of the second
fundamental form.
If the position vector $X$ evolves in the direction of the mean
curvature $H$,  then it gives rise to a solution to mean curvature flow:
$X(\cdot, t) : M^{n}\rightarrow \mathbb{R}^{n+1}$
satisfying $X(\cdot,0)=X(\cdot)$ and
$$ \frac{\partial X(p,t)}{\partial t}= H(p,t),\ \ \ (p,t)\in M \times [0,T),$$
where $H(p,t)$ denotes the mean curvature vector of hypersurface $M_{t}=X(M^{n},t)$
at point $X(p, t)$. In this section, we consider the self-shrinker
of the mean curvature flow, which is introduced by Huisken in \cite{H}(cf. Colding
and Minicozzi \cite{CM}).  An $n$-dimensional  submanifold $M^{n}$
in  the Euclidean space $\mathbb{R}^{n+p}$ is called a self-shrinker if it satisfies $$n\vec H=-X^N,$$ where $\vec H$ and
$X^N$ denote the mean curvature vector and  the orthogonal
projection of $X$  into the normal bundle of $M^n$, respectively.

\begin{thm}\label{thm2.1}
Under the assumption of theorem \ref{thm-product-soliton}, we have
\begin{equation}\label{6.1}
\lambda_{k+1}+ c
\leq (1+\frac {4}{n})(\lambda_{1}+c)\ k^{2/n},
\end{equation}
where $c$ is the same constant as in the theorem  \ref{thm-product-soliton}.
\end{thm}
\vskip 3mm
\textbf{\emph{Proof of theorem}}\ref{thm-shrinker}.
Since $M^n$ is a submanifold in the Euclidean space
$\mathbb{R}^{n+p}$, we have
\begin{equation*}
\Delta X=n\vec H.
\end{equation*}
Hence,
\begin{equation*}\begin{aligned}&
\Delta f= \langle X,\Delta X\rangle +n=n-n^2H^2,  \\&|\nabla
f|^2=|X|^2-|X^N|^2.\end{aligned}
\end{equation*}
Therefore, we have

\begin{equation}\begin{aligned}\label{eq-f-H}
n^{2}H^{2}+2|\Delta _{f}f|+|\nabla f|^{2}&=n^{2}H^{2}+|2n-|X|^2|+|X|^2-|X^N|^2\\&\leq n^{2}H^{2}+|2n-|X|^2|+|X|^2.\end{aligned}
\end{equation}
By \eqref{eq-f-H}, one can yield

\begin{equation}
\begin{aligned}\label{shringker-ineq}
\sum^{k}_{i=1}(\lambda_{k+1}-\lambda_{i})^{2}
&\leq\frac{4}{n}\sum^{k}_{i=1}(\lambda_{k+1}-\lambda_{i})
\left(\lambda_{i}+\frac{1}{4}\int_{M^{n}}u_{i}^{2}(2n-|X|^2)e^{-\frac{|X|^2}{2}}dv\right)\\&\leq\frac{4}{n}\sum^{k}_{i=1}(\lambda_{k+1}-\lambda_{i})
\left(\lambda_{i}+\frac{1}{4}\int_{M^{n}}u_{i}^{2}\left(n^{2}H^{2}+|2n-|X|^2|+|X|^2\right)e^{-\frac{|X|^2}{2}}dv\right)
\\&\leq\frac{4}{n}\sum^{k}_{i=1}(\lambda_{k+1}-\lambda_{i})
\left(\lambda_{i}+c\right),
\end{aligned}
\end{equation}where

$$c=\frac{1}{4}\inf_{\psi\in \Psi}\max_{\Omega}\left(n^{2}H^{2}+|2n-|X|^2|+|X|^2\right),$$and $\Psi$ denotes the set of all isometric immersions from $M^n$
into a Euclidean space. Here, we note that the first inequality \eqref{shringker-ineq} is established in  \cite{Z1}. Therefore, it follows from Cheng and Yang's recursion formula (see \cite{CY2}) that,
\begin{equation}
\begin{aligned}\label{shringker-upper}
\lambda_{k+1}+c
\leq C_{0}(n)\left(\lambda_{1}+c\right)k^{\frac{2}{n}},
\end{aligned}
\end{equation}Let $F_{j}(x)=\alpha_{j}x^{j}$ and $a_{j}>0$, such that

\begin{equation*}a_{j}^{2}=\|\nabla F_{j}u_{i}\|_{\Omega}^{2}\geq\sqrt{\||\nabla F_{j}|^{2}u_{i}\|_{\Omega}^{2}}=b_{j}\geq0,\end{equation*}

\begin{equation*}
\begin{aligned} \sum_{j=1}^{n+p}\int2 u_{i}\langle\nabla F_{j},\nabla f\rangle\Delta  F_{j}e^{-f}dv=0,\end{aligned}\end{equation*}
and

\begin{equation*}\begin{aligned}  \sum_{j=1}^{n+p}\int2 u_{i}\langle\nabla F_{j},\nabla u_{i}\rangle\Delta  F_{j}e^{-f}dv=0,\end{aligned}\end{equation*}
where   $j=1,2,\cdots,n+p,$ and $x^{j}$ denotes the $j$-th standard coordinate function of the Euclidean space $\mathbb{R}^{n+p}$.
Let $$\alpha=\min_{1\leq j\leq n+p}\{\alpha_{j}\},$$$$\overline{\alpha}=\max_{1\leq j\leq n+p}\{\alpha_{j}\},$$$$\beta=\min_{1\leq j\leq n+p}\{b_{j}\}.$$
According to theorem \ref{thm1.1}, lemma \ref{lem3.1} and \eqref{eq-f-H} , we have

\begin{equation}\begin{aligned}\sum_{j=1}^{l}\frac{a_{j}^{2}+b_{j}}{2}&=\sum_{j=1}^{n+p}\frac{a_{j}^{2}+b_{j}}{2}\\&\geq\frac{1}{2}\left(n\alpha^{2}+
\sum_{j=1}^{n+p}b_{j}\right)\\&\geq\frac{1}{2}\left(n\alpha^{2}+(n+p)\beta\right),\end{aligned}\end{equation}
and

\begin{equation}\begin{aligned}\sum_{j=1}^{n+p}\|2\langle\nabla F_{j},\nabla u_{i}\rangle+u_{i}\Delta_{f} F_{j}\|_{\Omega}^{2}&\leq\overline{\alpha}^{2}\left(4\lambda_{i}+\int_{\Omega}u^{2}_{i}\left(|\nabla f|^{2}+2|\Delta_{f}f|
+n^{2}H^{2}\right)e^{-f}dv\right)\\&\leq
4\overline{\alpha}^{2}\left(\lambda_{i}+c\right).\end{aligned}\end{equation}
Let $i=1,\tau=c$, then, by proposition \ref{prop2.3}, we
have

\begin{equation}\begin{aligned}\label{gap2}\left(\lambda_{k+2}-\lambda_{k+1}\right)^{2}
\leq\frac{32\overline{\alpha}^{2}(\lambda_{k+2}+c)}{n\alpha^{2}+(n+p)\beta}\left(\lambda_{1}+c\right)
,
\end{aligned}\end{equation}
Therefore, we deduce from \eqref{gap2} that,
\begin{equation*}\begin{aligned}\lambda_{k+2}-\lambda_{k+1}&\leq
\sqrt{\frac{32\overline{\alpha}^{2}}{n\alpha^{2}+(n+p)\beta}}\sqrt{\lambda_{1}+c}\sqrt{\lambda_{k+2}+c}\\&\leq
(\lambda_{1}+c)\sqrt{\frac{32\overline{\alpha}^{2}C_{0}(n)}{n\alpha^{2}+(n+p)\beta}}(k+1)^{\frac{1}{n}}
\\&=C_{n,\Omega}(k+1)^{\frac{1}{n}},\end{aligned}
\end{equation*}
where $$C_{n,\Omega}=(\lambda_{1}+c)\sqrt{\frac{32\overline{\alpha}^{2}C_{0}(n)}{n\alpha^{2}+(n+p)\beta}}.$$
Thus, we complete the proof
of this theorem.
$$\eqno\Box$$

\begin{rem}
Let $M^{n}$ be  an $n$-dimensional complete minimal self-shrinker  in the $(n+p)$-dimensional Euclidean space $\mathbb{R}^{n+p}$,
eigenvalues of the Dirichlet
problem \eqref{Eigenvalue-Problem} of drifting Laplacian with $f=\frac{|X|^2}{2}$ then the constant $c$ in theorem \ref{thm-z-6} will be written as$$c=\frac{1}{4}\inf_{\psi\in \Psi}\max_{\Omega}\left(|2n-|X|^2|+|X|^2\right).$$

\end{rem}

\begin{thm}\label{thm-product-soliton-1} Let $(M^{n},g)$ be an $n$-dimensional, compact Riemannian manifold. Let $H$ and $X$ denote the mean curvature of $M^n$ and the
position vector of $M^n$, respectively, and $\lambda_{i}$ be the $i$-th $(i=0,1,2,\cdots,k)$
eigenvalue of the eigenvalue problem \eqref{f-Laplacian-closed}. Then,  we have

\begin{equation*}\begin{aligned}\overline{\lambda}_{k+1}-\overline{\lambda}_{k}\leq
C_{n,\Omega}(k+1)^{\frac{1}{n}},\end{aligned}
\end{equation*}
where $$C_{n,\Omega}=(\overline{\lambda}_{1}+c)\sqrt{\frac{32\overline{\alpha}^{2}C_{0}(n)}{n\alpha^{2}+(n+p)\beta}},$$$$c=\frac{1}{4}\inf_{\psi\in \Psi}\max_{\Omega}\left(n^{2}H^{2}+|2n-|X|^2|+|X|^2\right),$$and $\Psi$ denotes the set of all isometric immersions from $M^n$
into a Euclidean space.
  \end{thm}
\begin{rem}Let $M^{n}$ be an $n$-dimensional complete self-shrinker with polynomial
volume growth in $\mathbb{R}^{n+1}$.     In \cite{CW}, Q.-M. Cheng and G. Wei proved the fact: If $M^{n}$ is noncompact manifolds with $A \leq 10/7
,$ which can  split into at most $m$ geodesic lines, then it is isometric to the
hyperplane $\mathbb{R}^{n}$ when $m=n$,
and $M^{n}$ is isometric to a cylinder $\mathbb{R}^{m}\times \mathbb{S}^{n-m}(\sqrt{n-m})$,
for $1\leq m \leq n-1$; If $M^{n}$ is compact, then it is isometric to the round sphere $\mathbb{S}^{n}(\sqrt{n})$.
Therefore, according to theorem \ref{thm-product-soliton-1} and the result of classification of self-shrinkers,
it is not difficult to estimate the upper bound for the gap of the consecutive eigenvalues Dirichlet problem \eqref{Eigenvalue-Problem}
of Laplacian with $f=\frac{|X|^{2}}{2}$.

\end{rem}

\begin{thm}Let $(M^{n},g)$ be an $n$-dimensional, compact Riemannian manifold and $\lambda_{i}$ be the $i$-th $(i=0,1,2,\cdots,k)$
eigenvalue of the eigenvalue problem \eqref{Eigen-Prob-closed} with $f=\frac{|X|^{2}}{4}$. Then, for any  $h\in C^{3}(\Omega)\cap C^{2}(\partial\Omega)$,  we have

\begin{equation*}\begin{aligned}\lambda_{k+1}-\lambda_{k}\leq
C_{n,\Omega}k^{\frac{1}{n}},\end{aligned}
\end{equation*}
where $$C_{n,\Omega}=4(\lambda_{1}+c_{1})\sqrt{\frac{C_{0}(n)}{n}}.$$ \end{thm}

\begin{rem}
Let $M^{2}$
be a $2$-dimensional complete self-shrinker in $\mathbb{R}^{3}$ with constant squared norm  of the second fundamental form,
Cheng and Ogata gave a complete classification as follows (see  \cite{CO}): $M^{2}$ is isometric to  $\mathbb{R}^{2}$,
or a cylinder $\mathbb{S}^{1}(1)\times \mathbb{R}$, or the round sphere $\mathbb{S}^{2}(2)$. Given some topological conditions( for example, the manifold can be
made such assumption that it is compact or can
be splitted into one geodesic line and so on), by theorem \ref{thm-product-soliton-1},
one can obtain the similar estimates for the gap of consecutive eigenvalues of Dirichlet problem \eqref{Eigenvalue-Problem} of drifting Laplacian with $f=\frac{|X|^{2}}{4}$
since eigenvalues is isometrically invariant.

\end{rem}

\section{Eigenvalues on the Complete Product Riemannian Manifolds}\label{sec7}

We let $M^{n}$ be an $n$-dimensional complete Riemannian manifold with
$\infty-$Bakry-Emery curvature ${\rm Ric}^{f}\geq0$,   and
$f\in C^{2} (M^{n})$ be bounded above uniformly on $M^{n}$. Under those assumption, F.-Q. Fang, X.-D. Li and Z.-L. Zhang \cite{FLZ} proved that it
splits isometrically as $\mathbb{N}^{n-m}\times \mathbb{R}^{m}$, where
$\mathbb{N}^{n-m}$ is some complete Riemannian manifold without lines and
$\mathbb{R}^{m}$ is the $m$-dimensional Euclidean space. Therefore, based on the above argument, we can prove the following:
\vskip 3mm

\begin{prop}\label{thm1-z}~Let $(M^{n},g,d\mu)$ be an $n$-dimensional complete
metric measure space with $\infty-$Bakry-Emery curvature ${\rm
Ric}^{f}\geq0$ and $f\in C^{2} (M^{n})$ be bounded above uniformly
on $M^{n}$. Assume that $\lambda_{i}$ is the $i$-th
$(i=1,2,\cdots,k)$ eigenvalue of the Dirichlet problem \eqref{Eigenvalue-Problem}, then there exists a positive integer $m$, where
$1\leq m\leq n$, such that

\begin{equation}\begin{aligned}\label{prod-eigen}\lambda_{k+1}-\lambda_{k}
\leq C(m,\Omega,k)k^{\frac{1}{m}},
\end{aligned}\end{equation}where $$C(m,\Omega,k)=\sqrt{\frac{4\left( C_{0}(m)\left(\lambda_{1}+c_{2}\right)\right)}{m}}
\cdot\sqrt{4\lambda_{1}+\max_{\Omega}\{|\nabla f|^{2}\}+4m\max_{\Omega}\{|\nabla f|\}\sqrt{C_{0}(m)\left(\lambda_{1}+c_{2}\right)}k^{\frac{1}{m}}},$$
and$$c_{2}=\frac{1}{4}\max_{\Omega}\{|\nabla f|^{2}\}+m\max_{\Omega}\{|\nabla f|\}\lambda_{k+1}^{\frac{1}{2}}.$$
\end{prop}

\begin{proof}Since the
$\infty-$Bakry-Emery curvature ${\rm Ric}^{f}$ is nonnegative and
$f\in C^{2} (M^{n})$ is bounded above uniformly on $M^{n}$, by
Theorem 1.1 in \cite{FLZ}, we know that the Bakry-\'{E}mery-Hadamard
manifold splits isometrically as $\mathbb{N}^{n-m}\times
\mathbb{R}^{m}$ , where $\mathbb{N}^{n-m}$ is some complete Riemannian
manifold without lines and $\mathbb{R}^{m}$ is the $m$-dimensional
Euclidean space. Since the eigenvalue of the Dirichlet problem is an
invariant of isometries, the remainder part of the proof is only to
show that the inequality \eqref{prod-eigen} holds on
Bakry-\'{E}mery-Hadamard product manifolds
$\mathbb{R}^{m}\times\mathbb{N}^{n-m}$. For any $j=1,2,\cdots,m$, let

\begin{equation*}
F_{j}=h_{j}(\mathbf{x},\mathbf{y})=h_{j}((x^{1},x^{2},\ldots,x^{m}),y)=x^{j},
\end{equation*}where $\mathbf{x}=(x^{1},x^{2},\ldots,x^{m})$ and $x^{j}$ is the $j$-th coordinate function. Then, we have
\begin{equation}\label{2.1}\Delta_{f}h_{j}(\mathbf{x},\mathbf{y})=\Delta x^{j}+\langle\nabla f,\nabla x^{p}\rangle=\langle\nabla f,\nabla x^{j}\rangle\leq|\nabla f|,\end{equation}
\begin{equation}\label{2.2} |\nabla h_{j}(\mathbf{x},\mathbf{y})|^{2}=1,\end{equation}
\begin{equation}\label{2.3}\sum_{p=1}^{m}\langle\nabla h_{j}(\mathbf{x},\mathbf{y}),\nabla u_{i}\rangle^{2}
\leq\sum_{j=1}^{m}\langle\nabla u_{i},\nabla u_{i}\rangle.\end{equation}
Hence, we have $|\nabla x^{j}|=1$ and $\Delta x^{j}=0$. Let $l=n$, then, from \eqref{general-formula-2}, we
have

\begin{equation*}a_{j}=\sqrt{\|\nabla F_{j}u_{i}\|_{\Omega}^{2}}=\sqrt{\|\nabla x_{j}u_{i}\|_{\Omega}^{2}}=1,\end{equation*}

\begin{equation*}b_{j}=\sqrt{\||\nabla F_{j}|^{2}u_{i}\|_{\Omega}^{2}}=\sqrt{\||\nabla x_{j}|^{2}u_{i}\|_{\Omega}^{2}}=1.\end{equation*}
Thus, we have
\begin{equation*}
a_{j}=b_{j}.\end{equation*}By the Cauchy-Schwarz inequality, we have

\begin{equation}\label{prod-7.8}\begin{aligned}\sum^{m}_{j=1}\|u_{i}\Delta_{f}h_{j}(\mathbf{x},\mathbf{y})+2\langle\nabla
h_{j}(\mathbf{x},\mathbf{y}),\nabla u_{i}\rangle\|_{\Omega}^{2}
\leq4\lambda_{i}+\max_{\Omega}\{|\nabla f|^{2}\}+4m\max_{\Omega}\{|\nabla f|\}\lambda_{i}^{\frac{1}{2}}.\end{aligned}\end{equation}
Let $l=m$, $F_{j}=h_{j}(\mathbf{x},\mathbf{y})$, by proposition \ref{prop2.3}, we have

\begin{equation*}\begin{aligned}\sum_{j=1}^{m}\frac{a_{j}^{2}+b_{j}}{2}\left(\lambda_{k+2}-\lambda_{k+1}\right)^{2}
\leq4(\lambda_{k+2}+\tau)\sum_{j=1}^{m}\|2\langle\nabla F_{j},\nabla u_{i}\rangle+u_{i}\Delta_{f}F_{j}\|^{2}
,
\end{aligned}\end{equation*}
which implies that

\begin{equation}\begin{aligned}\label{prod-7.9}m\left(\lambda_{k+2}-\lambda_{k+1}\right)^{2}
\leq4(\lambda_{k+2}+\tau)\sum_{j=1}^{m}\|2\langle\nabla F_{j},\nabla u_{i}\rangle+u_{i}\Delta_{f}F_{j}\|^{2}
.
\end{aligned}\end{equation}
From \eqref{prod-7.8} and \eqref{prod-7.9}, it is not difficult to see that,  for any $i=1,2,\cdots$,

\begin{equation}\begin{aligned}\label{prod-7.10}\lambda_{k+2}-\lambda_{k+1}
\leq\sqrt{\frac{4(\lambda_{k+2}+\tau)}{m}}\cdot\sqrt{4\lambda_{i}+\max_{\Omega}\{|\nabla f|^{2}\}+4m\max_{\Omega}\{|\nabla f|\}\lambda_{i}^{\frac{1}{2}}}
.
\end{aligned}\end{equation}Recall that the autor proved the following eigenvalue inequality in \cite{Z1}:
\begin{equation*}
\begin{aligned}
 \sum^{k}_{i=1}(\lambda_{k+1}-\lambda_{i})^{2}
\leq\frac{4}{m}\sum^{k}_{i=1}(\lambda_{k+1}-\lambda_{i})\left(\lambda_{i}+c_{1}\right),
\end{aligned}
\end{equation*}where $$c_{1}=\frac{1}{4}\max_{\Omega}\{|\nabla f|^{2}\}+m\max_{\Omega}\{|\nabla f|\}\lambda_{i}^{\frac{1}{2}}.$$Therefore, we have

\begin{equation}
\begin{aligned}\label{thm-z-6}
 \sum^{k}_{i=1}(\lambda_{k+1}-\lambda_{i})^{2}
\leq\frac{4}{m}\sum^{k}_{i=1}(\lambda_{k+1}-\lambda_{i})\left(4\lambda_{i}+c_{2}\right),
\end{aligned}
\end{equation}where $$c_{2}=\frac{1}{4}\max_{\Omega}\{|\nabla f|^{2}\}+m\max_{\Omega}\{|\nabla f|\}\lambda_{k+1}^{\frac{1}{2}}.$$
Therefore, by Q.-M. Cheng and H.-C. Yang's recursion formula and \eqref{thm-z-6}, we yield
\begin{equation}
\begin{aligned}\label{prod-upper}
 \lambda_{k+1}+c_{2}
\leq C_{0}(m)\left(\lambda_{1}+c_{2}\right)k^{\frac{2}{m}}.
\end{aligned}
\end{equation}Putting $\tau=c_{2}$,  and utilizing \eqref{prod-7.10} and \eqref{prod-upper}, then one can infer that

\begin{equation}\begin{aligned}\label{prod-7.11}\lambda_{k+2}-\lambda_{k+1}&
\leq\sqrt{\frac{4(\lambda_{k+2}+\tau)}{m}}\cdot\sqrt{4\lambda_{1}+\max_{\Omega}\{|\nabla f|^{2}\}+4m\max_{\Omega}\{|\nabla f|\}(\lambda_{k+2}+\tau)^{\frac{1}{2}}}
\\&
\leq\sqrt{\frac{4\left( C_{0}(m)\left(\lambda_{1}+c_{2}\right)(k+1)^{\frac{2}{m}}\right)}{m}}
\\&\cdot\sqrt{4\lambda_{1}+\max_{\Omega}\{|\nabla f|^{2}\}+4m\max_{\Omega}\{|\nabla f|\}\left( C_{0}(m)\left(\lambda_{1}+c_{2}\right)(k+1)^{\frac{2}{n}}\right)^{\frac{1}{2}}}
\\&
=\sqrt{\frac{4\left( C_{0}(m)\left(\lambda_{1}+c_{2}\right)\right)}{m}}
\\&\cdot\sqrt{4\lambda_{1}+\max_{\Omega}\{|\nabla f|^{2}\}+4m\max_{\Omega}\{|\nabla f|\}\sqrt{C_{0}(m)\left(\lambda_{1}+c_{2}\right)}(k+1)^{\frac{1}{m}}}
\cdot(k+1)^{\frac{1}{m}}\\&
=C(n,\Omega,k)(k+1)^{\frac{1}{m}},
\end{aligned}\end{equation}where $$C(m,\Omega,k)=\sqrt{\frac{4\left( C_{0}(m)\left(\lambda_{1}+c_{2}\right)\right)}{m}}
\cdot\sqrt{4\lambda_{1}+\max_{\Omega}\{|\nabla f|^{2}\}+4m\max_{\Omega}\{|\nabla f|\}\sqrt{C_{0}(m)\left(\lambda_{1}+c_{2}\right)}(k+1)^{\frac{1}{m}}}.$$
Hence, we complete the proof of the
proposition.
\end{proof}

If we replace the condition of ${\rm Ric}^{f}\geq0$ by
${\rm Ric}^{f}_{l,n}\geq0$, then the condition that $f$ bounded
above uniformly on $M^{n}$ can be removed.
Similarly, by using the same method as
proposition \ref{thm1-z} and noticing lemma 2.6 in \cite{FLZ}, it is
not difficult to  give the proof of the following proposition:
\begin{prop}\label{thm2-z}~Let $(M^{n},g,d\mu)$ be an $n$-dimensional, connected, complete
Bakry-\'{E}mery manifold with $l$-Bakry-\'{E}mery curvature ${\rm
Ric}^{f}_{l,n}\geq0$. Assume that $\lambda_{i}$ is the $i$-th
$(i=1,2,\cdots,k)$ eigenvalue of the Dirichlet problem \eqref{Eigenvalue-Problem}, then there exists a positive integer $m$, where
$1\leq m\leq n$, such that

\begin{equation*}\begin{aligned}\lambda_{k+1}-\lambda_{k}
\leq C(m,\Omega,k)k^{\frac{1}{m}},
\end{aligned}\end{equation*}where $$C(m,\Omega,k)=\sqrt{\frac{4\left( C_{0}(m)\left(\lambda_{1}+c_{2}\right)\right)}{m}}
\cdot\sqrt{4\lambda_{1}+\max_{\Omega}\{|\nabla f|^{2}\}+4m\max_{\Omega}\{|\nabla f|\}\sqrt{C_{0}(m)\left(\lambda_{1}+c_{2}\right)}k^{\frac{1}{m}}},$$
and$$c_{2}=\frac{1}{4}\max_{\Omega}\{|\nabla f|^{2}\}+m\max_{\Omega}\{|\nabla f|\}\lambda_{k+1}^{\frac{1}{2}}.$$
\end{prop}

\begin{rem}\label{rem-7.1}In the proofs of proposition \ref{thm1-z} and proposition \ref{thm2-z}, the weighted coefficients  are assumed that $a_{j}=1$
for any $j=1,2,\cdots,l$.\end{rem}

\begin{rem}\label{rem-7.1}Under the assumptions of proposition \ref{thm1-z} and proposition \ref{thm2-z}, Riemannian manifold $M^{n}$
splits isometrically as $\mathbb{N}^{n-m}\times \mathbb{R}^{m}$. Therefore, the integer $m$ in proposition \ref{thm1-z} and proposition \ref{thm2-z} is exactly the dimension of
the Euclidean space $\mathbb{R}^{m}$.\end{rem}
\begin{rem}\label{rem-7.2}Suppose that   $(M^{n},g,f)~(n\geq4)$, is a complete, shrinking, gradient Ricci soliton with harmonic Weyl tenson,
then, $M^{n}=N^{k}\times\mathbb{R}^{n-k}$, where $N^{k}$ is an Einstein manifold (cf. \cite{FG,MS}).
Therefore, by the same method as the proof of proposition \ref{thm1-z}, it is not difficult to
obtain a similar estimate for the consecutive eigenvalues of drifting Laplacian on the soliton.\end{rem}

In order to generalize the trivial Ricci solitons, Petersen and Wylie  introduced the notion of
rigidity of gradient Ricci solitons in \cite{PW1}. A gradient soliton is said to be rigid if it is isometric to a
quotient of $\mathbb{N} \times \mathbb{R}^{k}$ where $\mathbb{N}$ is an Einstein manifold and $f = \frac{\rho}{2}|x|^{2}$ on the Euclidean factor.
That is, the Riemannian manifold $(M^{n}, g)$  is isometric to $\mathbb{N}\times_{\Gamma} \mathbb{R}^{k}$, where $\Gamma$ acts freely on $N$
and by orthogonal transformations on $\mathbb{R}^{k}$. Rigidity of gradient Ricci solitons has been studied in \cite{PW1,PW2}.

\begin{rem}
It is well known that Einstein manifolds have harmonic Weyl tensor. In fact, under some geometric implications,  a Ricci soliton has the assumption of
the harmonicity of the Weyl tensor. For example, F.-L. Manuel and G.-R. Eduardo {\rm\cite{ME2}} showed that
a compact Ricci soliton is rigid if and only if it has harmonic Weyl tensor, which gives a positive answer to Problem C.2 posed in \cite{ENM}. For the complete
noncompact case, F.-L. Manuel and G.-R. Eduardo proved that a gradient shrinking Ricci soliton is rigid if and only if it
has harmonic Weyl tensor, under the assumptions that the Ricci curvature is bounded from
below and the Riemannian curvature has at most exponential growth in \cite{ME2}. Therefore, by remark \ref{rem-7.2} and
the proof of  proposition  \ref{thm1-z}, one can obtain a similar estimate for
the consecutive eigenvalues of drifting Laplacian on those solitons with the above rigid and geometric conditions.
Let $(M^{n}, g,f)$ be an $n$-dimensional compact Ricci soliton with constant sectional curvature. Then, the Weyl tensor
vanishes \cite{ME2}. Therefore, for all of the compact Ricci soliton with constant sectional curvature, one can also obtain the similar eigenvalue inequality by the same argument. In addition,
by the other classifications of Ricci solitons, for example in \cite{MS,N}, one can obtain the corresponding eigenvalue inequality of drifting Laplacian on some
complete metric measure spaces.
\end{rem}

 If we consider the case that  $f$ is a constant, the drifting Laplacian is exactly the standard Laplacian on complete
Riemannian manifolds. Then, one can prove the following:
\begin{corr} Let $(M^{n},g,d\mu)$ be an $n$-dimensional complete
Riemannian manifold with Ricci curvature ${\rm
Ric} \geq0$ and $f\in C^{2} (M^{n})$ be bounded above uniformly
on $M^{n}$. Assume that $\lambda_{i}$ is the $i$-th
$(i=1,2,\cdots,k)$ eigenvalue of the Dirichlet problem \eqref{Eigen-Prob-Lapl}, then there exists a positive integer $m$, where
$1\leq m\leq n$, such that

\begin{equation*}\begin{aligned}\lambda_{k+1}-\lambda_{k}
\leq C(m,\Omega,k)k^{\frac{1}{m}},
\end{aligned}\end{equation*}where $$C(m,\Omega,k)=4\lambda_{1}\sqrt{\frac{   C_{0}(m)}{m}}
.$$
\end{corr}

\section{Appendix}
\vskip 3mm

In this appendix, we give the proof of theorem \ref{thm-product-soliton}.
\vskip 3mm

\textbf{\emph{Proof of theorem}} \ref{thm-product-soliton}. We denote the position vector of the $n$-dimensional unit round
cylinder $\mathbb{R}^{n-m}\times\mathbb{S}^{m}(1)$ in
$n+1$-dimensional Euclidean space $\mathbb{R}^{n+1}$ by $$\textbf{x}
=(\textbf{v},~\textbf{w})= (x^{1},
x^{2},\ldots,x^{n-m},x^{n-m+1},x^{n-m+2} \cdots,x^{n}, x^{n+1}),$$
where $\textbf{v}=(x^{1},
x^{2},\ldots,x^{n-m}),\textbf{w}=(x^{n-m+1},x^{n-m+2} \cdots,x^{n},
x^{n+1})$, and then, we obtain
\begin{equation}\label{4.3.54}\sum^{n+1}_{j=n-m+1}(x^{j})^{2}=1,~ \sum^{n+1}_{j=1}|\nabla x^{j}|^{2}=n, \end{equation} and
\begin{equation}\label{4.3.55}\Delta x^{j}=
\left\{ \begin{aligned}
     &0, &\textnormal{if}\ \ j=1,\cdots,n-m,   \\
                  &-mx^{j},  &\textnormal{if}\ \ j=n-m+1,\cdots,n+1.
                          \end{aligned} \right.
                          \end{equation}
For any $j~(j=1,2,\cdots,n+1)$,   let $l=n+1$ and
$F_{j}(x)=\delta_{j}x^{j}$ and $\delta_{j}>0$, such that

\begin{equation}\begin{aligned}\sum^{n+1}_{j=1}\int_{\Omega}u_{i}^{2}\Delta
(\delta_{j}x^{j})\langle \nabla
\left(\frac{(m-1)|\textbf{v}|^{2}_{\mathbb{R}^{n-m}}}{2}\right),\nabla
(\delta_{j}x^{j})\rangle
d\mu=0,\end{aligned}\end{equation}

\begin{equation}\begin{aligned}&(m-1)\sum^{n-m}_{j=1}\int_{\Omega}\langle\nabla
(\delta_{j}x^{j}),\nabla u_{i}\rangle
u_{i}(\delta_{j}x^{j})d\mu+m\sum^{n+1}_{j=n-m+1}\int_{\Omega}\langle\nabla (\delta_{j}x^{j}),\nabla
u_{i}\rangle u_{i}(\delta_{j}x^{j})d\mu\\&=\widetilde{\delta}^{2}\left[(m-1)\sum^{n-m}_{j=1}\int_{\Omega}\langle\nabla
x^{j},\nabla u_{i}\rangle
u_{i}x^{j}d\mu-4m\sum^{n+1}_{j=n-m+1}\int_{\Omega}\langle\nabla x^{j},\nabla
u_{i}\rangle u_{i}x^{j}d\mu\right],\end{aligned}\end{equation} and

\begin{equation*}a_{j}^{2}=\|\nabla F_{j}u_{i}\|^{2}\geq\sqrt{\||\nabla F_{j}|^{2}u_{i}\|^{2}}=b_{j}\geq0.\end{equation*} Let $$\delta=\min_{1\leq j\leq n+p}\{\delta_{j}\},$$$$\overline{\delta}=\max_{1\leq j\leq n+p}\{\delta_{j}\},$$
$$\gamma=\min_{1\leq j\leq n+p}\min_{\Omega}\{b_{j}\}.$$
Then, we have

\begin{equation}\begin{aligned}\label{eq-a-b}\sum_{j=1}^{l}\frac{a_{j}^{2}+b_{j}}{2}&=\sum_{j=1}^{n+1}\frac{\sqrt{\|\nabla (\delta_{j}x^{j})u_{i}\|^{2}}+\sqrt{\||\nabla (\delta_{j}x^{j})|^{2}u_{i}\|^{2}}}{2}\\&\geq\frac{1}{2}\left(n\delta^{2}+\sum^{n+1}_{j=1}b_{j}\right)
\\&\geq\frac{1}{2}\left(n\delta^{2}+(n+1)\gamma\right).
\end{aligned}\end{equation}
For any fixed point $x_{0}\in\Omega$, we can find a coordinate
system $(\widetilde{x}^{1},\widetilde{x}^{2},\cdots
\widetilde{x}^{n+1})$ of the $n$-dimensional unit round cylinder
$\mathbb{R}^{n-m}\times\mathbb{S}^{m}(1)$ such that at the point
$x_{0}$

\begin{equation}\label{4.3.58}\begin{aligned}&\widetilde{x}^{1}=\cdots=\widetilde{x}^{n}=0,~~\widetilde{x}^{n+1}=1,
\\&\nabla\widetilde{x}^{n+1}=0;\\&\nabla_{p}x^{q}=\delta^{q}_{p}~(p,q=1,2,\cdots,n+1).
\end{aligned}\end{equation}
In fact, we can choose a constant $(n+1)\times(n+1)$ type
orthonormal matrix $(a^{i}_{j})_{(n+1)\times(n+1)}$
satisfying\begin{equation*}\begin{aligned}\sum^{n+1}_{\alpha=1}a^{\alpha}_{p}a^{\alpha}_{q}=\delta_{pq},
\end{aligned}\end{equation*}
such that
\begin{equation*}\begin{aligned}x^{p}=\sum^{n+1}_{\alpha=1}a^{p}_{\alpha}\widetilde{x}^{\alpha},
\end{aligned}\end{equation*}
and (\ref{4.3.58}) is satisfied at the point $x_{0}$. Thus, at the
point $x_{0}$, we
have\begin{equation*}\begin{aligned}\sum^{n+1}_{p=1}\langle\nabla
x^{p},\nabla
u_{i}\rangle^{2}&=\sum^{n+1}_{p,q,\alpha=1}a^{\alpha}_{p}a^{\alpha}_{q}\langle\nabla\widetilde{x}^{p},\nabla
u_{i}\rangle\langle\nabla\widetilde{x}^{q},\nabla u_{i}\rangle
\\&=\sum^{n+1}_{p=1}\langle\nabla\widetilde{x}^{p},\nabla u_{i}\rangle^{2}
\\&=\sum^{n+1}_{p=1}\langle\nabla_{p}u_{i},\nabla_{p}u_{i}\rangle\\&=|\nabla
u_{i}|^{2}.\end{aligned}\end{equation*} Since $x_{0}$ is an
arbitrary point, we know that for any point $x\in\Omega$,
\begin{equation*}\begin{aligned}\sum^{n+1}_{p=1}\langle\nabla
x^{p},\nabla u_{i}\rangle^{2}=|\nabla
u_{i}|^{2}.\end{aligned}\end{equation*}
On the other hand, by using
(\ref{4.3.54}), we
have\begin{equation}\label{4.3.59}\begin{aligned}&\sum^{n+1}_{p=n-m+1}\nabla(x^{p})^{2}=0,\end{aligned}\end{equation}and
\begin{equation}\label{4.3.60}\begin{aligned}\sum^{n+1}_{p=n-m+1}|\nabla
x^{p}|^{2}=-\sum^{n+1}_{p=1}x^{p}\Delta x^{p}=m.
\end{aligned}\end{equation}
Let

\begin{equation}\begin{aligned}\label{a-ineq}\mathfrak{A}=\sum_{j=1}^{l}\|2\langle\nabla F_{j},\nabla u_{i}\rangle+u_{i}\Delta_{f}F_{j}\|^{2}&=\sum_{j=1}^{n+1}\|2\langle\nabla (\delta_{j}x^{j}),\nabla u_{i}\rangle+u_{i}\Delta_{f}(\delta_{j}x^{j})\|^{2}.\end{aligned}\end{equation}
Then, using \eqref{4.3.59} and \eqref{4.3.60}, we deduce

\begin{equation}\label{4.3.62}\begin{aligned}\mathfrak{A}&=\sum^{n+1}_{j=1} \|2\langle\nabla
(\delta_{j}x^{j}),\nabla u_{i}\rangle+u_{i}\Delta (\delta_{j}x^{j})-u_{i}\nabla
\left(\frac{(m-1)|\textbf{v}|^{2}_{\mathbb{R}^{n-m}}}{2}\right),\nabla
(\delta_{j}x^{j})\rangle\|_{\Omega}^{2}\\
&=4\sum^{n+1}_{j=1}\int_{\Omega}\langle\nabla
(\delta_{j}x^{j}),\nabla
u_{i}\rangle^{2}d\mu+m^{2}\sum^{n+1}_{j=n-m+1}\int_{\Omega}u_{i}^{2}(\delta_{j}x^{j})^{2}d\mu\\
&+(m-1)^{2}\sum^{n-m}_{j=1}\int_{\Omega}u_{i}^{2}(\delta_{j}x^{j})^{2}d\mu
-2\sum^{n+1}_{j=1}\int_{\Omega}u_{i}^{2}\Delta
(\delta_{j}x^{j})\langle \nabla
\left(\frac{(m-1)|\textbf{v}|^{2}_{\mathbb{R}^{n-m}}}{2}\right),\nabla
(\delta_{j}x^{j})\rangle
d\mu\\
&-4(m-1)\sum^{n-m}_{j=1}\int_{\Omega}\langle\nabla
(\delta_{j}x^{j}),\nabla u_{i}\rangle
u_{i}(\delta_{j}x^{j})d\mu-4m\sum^{n+1}_{j=n-m+1}\int_{\Omega}\langle\nabla (\delta_{j}x^{j}),\nabla
u_{i}\rangle u_{i}(\delta_{j}x^{j})d\mu.\end{aligned}\end{equation}
Furthermore, by the definitions of $\overline{\delta}$ and $\widetilde{\delta}$, we have

\begin{equation}\label{4.3.62}\begin{aligned}\mathfrak{A}
&\leq4\overline{\delta}^{2}\sum^{n+1}_{j=1}\int_{\Omega}\langle\nabla
x^{j},\nabla
u_{i}\rangle^{2}d\mu+m^{2}\overline{\delta}^{2}\sum^{n+1}_{j=n-m+1}\int_{\Omega}u_{i}^{2}(x^{j})^{2}d\mu\\
&+(m-1)^{2}\overline{\delta}^{2}\sum^{n-m}_{j=1}\int_{\Omega}u_{i}^{2}(x^{j})^{2}d\mu
-2\sum^{n+1}_{j=1}\int_{\Omega}u_{i}^{2}\Delta
(\delta_{j}x^{j})\langle \nabla
\left(\frac{(m-1)|\textbf{v}|^{2}_{\mathbb{R}^{n-m}}}{2}\right),\nabla
(\delta_{j}x^{j})\rangle
d\mu\\
&-4(m-1)\sum^{n-m}_{j=1}\int_{\Omega}\langle\nabla
(\delta_{j}x^{j}),\nabla u_{i}\rangle
u_{i}(\delta_{j}x^{j})d\mu-4m\sum^{n+1}_{j=n-m+1}\int_{\Omega}\langle\nabla (\delta_{j}x^{j}),\nabla
u_{i}\rangle u_{i}(\delta_{j}x^{j})d\mu\\
&=4\overline{\delta}^{2}\lambda_{i}+m^{2}\overline{\delta}^{2}\sum^{n+1}_{j=n-m+1}\int_{\Omega}u_{i}^{2}(x^{j})^{2}d\mu
+(m-1)^{2}\overline{\delta}^{2}\sum^{n-m}_{j=1}\int_{\Omega}u_{i}^{2}(x^{j})^{2}d\mu
\\
&-(m-1)\widetilde{\delta}^{2}\sum^{n-m}_{j=1}\int_{\Omega}\langle\nabla (x^{j})^{2},\nabla
(u_{i})^{2}\rangle d\mu-m\widetilde{\delta}^{2}\sum^{n+1}_{j=n-m+1}\int_{\Omega}\langle\nabla (x^{j})^{2},\nabla
(u_{i})^{2}\rangle d\mu\\
&=4\overline{\delta}^{2}\lambda_{i}+(m-1)^{2}\overline{\delta}^{2}\sum^{n+1}_{j=1}\int_{\Omega}u_{i}^{2}(x^{j})^{2}d\mu
+(2m-1)\overline{\delta}^{2}\sum^{n+1}_{j=n-m+1}\int_{\Omega}u_{i}^{2}(x^{j})^{2}d\mu\\
&-(m-1)\widetilde{\delta}^{2}\sum^{n+1}_{j=1}\int_{\Omega}\langle\nabla (x^{j})^{2},\nabla
(u_{i})^{2}\rangle d\mu-\widetilde{\delta}^{2}\sum^{n+1}_{j=n-m+1}\int_{\Omega}\langle\nabla (x^{j})^{2},\nabla
(u_{i})^{2}\rangle d\mu\\
&=4\lambda_{i}+\mathfrak{B}+(2m-1)\overline{\delta}^{2},\end{aligned}\end{equation} where

\begin{equation}\begin{aligned}\label{4.3.63}\mathfrak{B}&=(m-1)^{2}\overline{\delta}^{2}\sum^{n+1}_{j=1}\int_{\Omega}u_{i}^{2}(x^{j})^{2}d\mu
-(m-1)\widetilde{\delta}^{2}\sum^{n+1}_{j=1}\int_{\Omega}\langle\nabla (x^{j})^{2},\nabla
(u_{i})^{2}\rangle d\mu\\&-\widetilde{\delta}^{2}\sum^{n+1}_{j=n-m+1}\int_{\Omega}\langle\nabla (x^{j})^{2},\nabla
(u_{i})^{2}\rangle d\mu\\
&=(m-1)\sum^{n+1}_{j=1}\int_{\Omega}\left((m-1)\overline{\delta}^{2}u_{i}^{2}(x^{j})^{2}
-\widetilde{\delta}^{2}\langle\nabla (x^{j})^{2},\nabla (u_{i})^{2}\rangle\right) d\mu\\
&=(m-1)\int_{\Omega}u_{i}^{2}\left((m-1)\overline{\delta}^{2}|\textbf{x}|^{2}
+\widetilde{\delta}^{2}\Delta_{f}|\textbf{x}|^{2}\right) d\mu.
\end{aligned}\end{equation}Uniting (\ref{4.3.54}), (\ref{4.3.55}), (\ref{4.3.59}) and (\ref{4.3.60}), we have
$$\sum^{n+1}_{p=1}\Delta(x^{p})^{2}=2(n-m).$$ By a direct computation, we yield
\begin{equation}\label{4.3.64}\begin{aligned}\Delta_{f}|\textbf{x}|^{2}=2(n-1)-2(m-1)|\textbf{x}|^{2}.\end{aligned}\end{equation}
Substituting
(\ref{4.3.62}), (\ref{4.3.63})  and (\ref{4.3.64})
into \eqref{a-ineq}, we obtain

\begin{equation}\label{factor-ineq}\begin{aligned}&\sum^{n+1}_{p=1} \|2\langle\nabla
x^{p},\nabla u_{i}\rangle+u_{i}\Delta x^{p}-u_{i}\langle \nabla
\left(\frac{(m-1)|\textbf{v}|^{2}_{\mathbb{R}^{n-m}}}{2}\right),\nabla
x^{p}\rangle\|_{\Omega}^{2}
\\&=(m-1)\int_{\Omega}u_{i}^{2}\left((m-1)\overline{\delta}^{2}|\textbf{x}|^{2}+\widetilde{\delta}^{2}[2(n-1)-2(m-1)|\textbf{x}|^{2}]\right)d\mu
+(2m-1)\overline{\delta}^{2}+4\overline{\delta}^{2}\lambda_{i}
\\&=(m-1)^{2}\left(\overline{\delta}^{2}-2\widetilde{\delta}^{2}\right)\int_{\Omega}u_{i}^{2}|\textbf{x}|^{2}d\mu
+(m-1)(2n-1)\widetilde{\delta}^{2}+(2m-1)\overline{\delta}^{2}+4\overline{\delta}^{2}\lambda_{i}
\\&\leq(m-1)^{2}\max_{\Omega}\left(\overline{\delta}^{2}-2\widetilde{\delta}^{2}\right)
|\textbf{x}|^{2}+(m-1)(2n-1)\widetilde{\delta}^{2}+(2m-1)\overline{\delta}^{2}+4\overline{\delta}^{2}\lambda_{i}
\\&\leq(m-1)^{2}\left(\overline{\delta}^{2}+2\widetilde{\delta}^{2}\right)+(m-1)(2n-1)\widetilde{\delta}^{2}+(2m-1)\overline{\delta}^{2}+4\overline{\delta}^{2}\lambda_{i}.
\end{aligned}\end{equation}
On the other hand, we have

\begin{equation}\begin{aligned} .\end{aligned}\end{equation}
By the recursion formula given by Q.-M. Cheng and H.-C. Yang in \cite{CY@}, we have
Let

\begin{equation*}\begin{aligned}c&=\frac{1}{4\overline{\delta}^{2}}\left[(m-1)^{2}\left(\overline{\delta}^{2}+2\widetilde{\delta}^{2}\right)
+(m-1)(2n-1)\widetilde{\delta}^{2}+(2m-1)\overline{\delta}^{2} \right]\\&
=\frac{(m-1)^{2}\left(\overline{\delta}^{2}+2\widetilde{\delta}^{2}\right)}{4\overline{\delta}^{2}}
+\frac{(m-1)(2n-1)\widetilde{\delta}^{2}}{4\overline{\delta}^{2}}+\frac{(2m-1)}{4}.\end{aligned}\end{equation*} Then, we deduce from \eqref{eq-a-b} and \eqref{general-formula-2} that,

\begin{equation}\begin{aligned}\label{general-rmula-3}\frac{1}{2}\left(n\delta^{2}+(n+1)\gamma\right)\left(\lambda_{k+2}-\lambda_{k+1}\right)^{2}
\leq4(\lambda_{k+2}+\rho)\sum_{j=1}^{l}\|2\langle\nabla F_{j},\nabla u_{i}\rangle+u_{i}\Delta_{f}F_{j}\|_{\Omega}^{2}
.
\end{aligned}\end{equation}Let $\tau=c,l=n+1$. Then, by utilizing \eqref{factor-ineq} and \eqref{general-rmula-3}, we yield

\begin{equation}\begin{aligned}\left(\lambda_{k+2}-\lambda_{k+1}\right)^{2}\leq\frac{32\overline{\delta}^{2}}{\left(n\delta^{2}+(n+1)\gamma\right)}(\lambda_{k+2}+c)(\lambda_{1}+c).
\end{aligned}\end{equation}
Therefore, we yield

\begin{equation*}\begin{aligned}\lambda_{k+2}-\lambda_{k+1}&\leq\sqrt{\frac{32\overline{\delta}^{2}}{\left(n\delta^{2}+(n+1)\gamma\right)}}\sqrt{\lambda_{k+2}+c}\sqrt{\lambda_{1}+c}\\
&\leq(\lambda_{1}+c)\sqrt{\frac{32C_{0}(n)\overline{\delta}^{2}}{\left(n\delta^{2}+(n+1)\gamma\right)}}(k+1)^{\frac{1}{n}}
\\&=C_{n,\Omega}(k+1)^{\frac{1}{n}},\end{aligned}
\end{equation*}
where $$C_{n,\Omega}=(\lambda_{1}+c)\sqrt{\frac{32C_{0}(n)\overline{\delta}^{2}}{\left(n\delta^{2}+(n+1)\gamma\right)}}.$$This completes the
proof of this theorem.
$$\eqno\Box$$

\begin{ack} The author is supported by the National Nature Science Foundation of China (Grant No. 11401268).\end{ack}


\begin{thebibliography}{99}

\bibitem{AN}
B. Andrews and L. Ni, {\it Eigenvalue comparison on Bakry-Emery manifolds,}
Communications in Partial Differential Equations, (2012), \textbf{37} (11): 2081-2092

\bibitem{A1}
M. S. Ashbaugh, {\it Isoperimetric and universal inequalities for eigenvalues}, in Spectral theory and
geometry (Edinburgh,1998), E. B. Davies and Yu Safalov eds., London Mathematical Society Lecture Notes,
 {\bf 273} (1999), Cambridge University Press, Cambridge: 95-139.


\bibitem{A2}
M. S. Ashbaugh, {\it  Universal eigenvalue bounds of Payne-Polya-Weinberger, Hile-Prottter, and H.C.
Yang,} Proceedings of the Indian Academy of Sciences-Mathematical Sciences. Springer India, 2002, {\bf 112}(1): 3-30.

\bibitem{BE}
W. Ballmann and P. Eberlein, {\it Fundamental groups of manifolds of
nonpositive curvature}, Journal of Differential Geometry, 1987, {\bf 25}(1): 1-22.

\bibitem{Be}
P.H. B\'{e}rard, {\it Spectral geometry: direct and inverse problems,}  Monografi\'{\i}as de Matem\'{a}tica
41, Instituto de Matem\'{a}tica Pura e Aplicada, Rio de Janeiro, 1986.

\bibitem{B}
K. Brighton, {\it A Liouville-type theorem for smooth metric measure
spaces}, Journal of Geometric Analysis, 2013, \textbf{23}(2):  562-570

\bibitem{CH2}
H.-D. Cao, {\it  Recent progress on Ricci solitons}, 2010, Adv. Lect. Math. (ALM), 11 Int. Press,
Somerville, MA: 1-38.

\bibitem{CaZ}
H.-D. Cao and D. Zhou, {\it On complete gradient shrinking Ricci
solitons}, Journal of Differential Geometry, 2010, {\bf85} (2): 175-185.

\bibitem{CX1}
X.-D. Cao, {\it Closed gradient shrinking Ricci solitons with positive
curvature operator,} The Journal of Geometric Analysis, 2007, {\bf17}(3): 451-459.

\bibitem{CLR}
N. Charalambous, Z. Lu and J. Rowlett, {\it Eigenvalue estimates on Bakry-\'{E}mery
manifolds,} Elliptic and Parabolic Equations, (2015), Volume {\bf 119} of the series Springer Proceedings in Mathematics \& Statistics: 45-61.

\bibitem{CL}
N. Charalambous, Z. Lu. {\it Heat kernel estimates and the essential spectrum on weighted manifolds,} The Journal of Geometric Analysis,  2015, {\bf 25}(1): 536-563.

\bibitem{Chavel}
I. Chavel, {\it Eigenvalues in Riemannian Geometry}, Academic Press, New
York, 1984.

\bibitem{CC}
D. Chen and Q.-M. Cheng, {\it Extrinsic estimates for estimates for eigenvalues of the Laplacian operator,}  Journal of the Mathematical Society of Japan, (2008), {\bf 60}(2): 325-339.

\bibitem{CZY}
D. Chen, T. Zheng and H.-C. Yang, {\it Estimates of the gaps between consecutive eigenvalues of Laplacian,} Pacific Journal of Mathematics, 2016, {\bf 282} (2): 293-311.

\bibitem{CO}
Q.-M. Cheng and S. Ogata, {\it 2-dimensional complete self-shrinkers in $\mathbb{R}^3$,} arxiv.org/abs/1504.02225v1

\bibitem{CP}
Q.-M. Cheng and Y. Peng, {\it  Estimates for eigenvalues of $\mathfrak L$
operator on self-shrinkers,} Communications in Contemporary Mathematics, 2013, \textbf{15}(06): 1350011 (23 pages), IDO:10.1142/S0219199713500119.

\bibitem{CW}
Q.-M., Cheng and G. Wei,  {\it  A gap theorem of self-shrinkers,} Trans. Amer. Math. Soc., 2015, {\bf367}(7):4895-4915.

\bibitem{CY1}
Q. -M. Cheng and H.-C. Yang, {\it  Estimates on eigenvalues of
Laplacian,} Mathematische Annalen, 2005, \textbf{331}(2): 445-460

\bibitem{CY2}
Q.-M. Cheng and H.-C. Yang, {\it  Bounds on eigenvalues of Dirichlet
Laplacian,} Mathematische Annalen, 2007, \textbf{337} (1):  159-175

\bibitem{CZ}
Q.-M. Cheng and L. Zeng, {\it  Eigenvalues of the Witten-Laplacian on compact Riemannian manifolds,} preprint.
\bibitem{C}
S.Y. Cheng. {\it  Eigenvalue comparison theorems and its geometric applications,} Mathematische Zeitschrift, 1975, {\bf 143}(3): 289-297.

\bibitem{CK}
B. Chow and D. Knopf, {\it  The Ricci flow: An introduction}, mathematical
surverys and monographs, 2004, {\bf 110}, American Mathematical Society.

\bibitem{CM}
T. H. Colding and W. P. Minicozzi II, {\it  Generic mean curvature flow I;
Generic Singularities}, Annals of Mathematics, 2012, {\bf 175} (2): 755-833.

\bibitem{ENM}
M. Eminenti, G. La Nave, C. Mantegazza, {\it Ricci solitons-the equation point of view,} manuscripta mathematica, 2008, {\bf 127}(3), 345-367.

\bibitem{FLZ}
F.-Q. Fang, X.-D. Li and Z.-L. Zhang, {\it  Two generalizations of
Cheeger-Gromoll splitting theorem via Bakry-\'{E}mery Ricci
curvature}, Annales de l'Institut Fourier, 2009, {\bf 59} (2): 563-573.

\bibitem{FG}
M. Fern\'{a}ndez-L\'{o}pez and E. Garc\'{\i}a-R\'{\i}o, A remark on compact Ricci solitons, Mathematische Annalen, 2008, {\bf 340} (4):893-896.

\bibitem{FS}
K. Funano, T. Shioya. {\it  Concentration, Ricci curvature, and eigenvalues of Laplacian.}  Geometric and Functional Analysis, 2013, {\bf 23} (3): 888-936.

\bibitem{FuS}
A. Futaki and Y. Sano, Lower diameter bounds for compact shrinking
Ricci solitons,  Asian Journal of Mathematics, 2013, {\bf 17}(1): 17-32.
\bibitem{FLL}
A. Futaki, H. Li and X.-D. Li, {\it  On the first eigenvalue of the
Witten-Laplacian and the diameter of compact shrinking Ricci
solitons,} Annals of Global Analysis and Geometry, 2013 , \textbf{44} (2): 105-114

\bibitem{Gro}
M. Gromov. {\it  Metric structures for Riemannian and non-Riemannian spaces,} Reprint of the
2001 English edition, Modern Birkh\"{a}user Classics, Birkh\"{a}user Boston Inc., Boston, MA,  Springer Science \& Business Media,
2007.

\bibitem{Ha1}
R.S. Hamilton, {\it  Three manifolds with positive Ricci curvature,} ]. Journal of Differential Geometry, 1982, \textbf{17} (2): 255-306

\bibitem{Ha2}
R.S. Hamilton, {\it  The Ricci flow on surfaces,} Mathematics and general relativity (Santa Cruz,
CA, 1986), 237-262,  Contemporary Mathematics, \textbf{71}, American  Mathematical Society, Providence, RI, 1988.


\bibitem{Ha3}
R. S. Hamilton, {\it  The formation of singularities in the Ricci flow},
Surveys in Differential Geometry (Cambridge, MA, 1993), {\bf 2}, 1995,
International Press, Combridge, MA: 7-136.

\bibitem{H1}
E. M. Harrell II, {\it  Some geometric bounds on eigenvalue gaps}, Communications in Partial Differential Equations, 1993,
{\bf 18}(1-2): 179-198.

\bibitem{HM}
E. M. Harrell II and P. L. Michel, {\it  Commutator bounds for eigenvalues with applications to spectral
geometry,}  Communications in Partial Differential Equations, 1994, {\bf 19}(11-12): 2037-2055.

\bibitem{H3}
E. M. Harrell II and J. Stubbe,  {\it  On trace identities and universal eigenvalue estimates for some
partial differential operators,} Transactions of the American Mathematical Society, 1997, {\bf 349}(5): 1797-1809.

\bibitem{H4}
E. M. Harrell II, {\it  Commutator,  eigenvalue  gaps  and mean  curvature in  the theory  of  Schr\"{o}inger
operators,} Communications in Partial Differential Equations, 2007, {\bf 32}(3): 401-413.

\bibitem{Has}
A. Hassannezhad, {\it  Eigenvalues of perturbed Laplace operators on compact manifolds,} Pacific Journal of Mathematics, 2013, {\bf 264}(2): 333-354.

\bibitem{HZ}
C. He and M. Zhu, {\it  Ricci solitons on Sasakian manifolds,} arxiv.org/abs/1109.4407.

\bibitem{HP}
G.N. Hile, M.H. Protter, {\it Inequalities for eigenvalues of the Laplacian},  Indiana University Mathematics Journal, 1980, {\bf 29}(4): 523-538.

\bibitem{H}
G. Huisken, {\it  Asymptotic behavior for singularities of the mean
curvature flow},  Journal of Differential Geometry, 1990, {\bf 31} (1), 285-299.

\bibitem{IM1}
 S. Ilias and O. Makhoul,  {\it''Universal'' inequalities for the eigenvalues of the Hodge de
Rham Laplacian,}  Annals of Global Analysis and Geometry, 2009, {\bf 36} (2): 191-204.

\bibitem{IM2}
 S. Ilias and O. Makhoul,  {\it A generalization of a Levitin and Parnovski Universal inequality
for eigenvalues}, 2012, Journal of Geometric Analysis, {\bf 22}(1): 206-222.

\bibitem{Iv}
T. Ivey, {\it Ricci solitons on compact three-manifolds,}  Differential Geometry and its Applications, 1993, {\bf 3}(4): 301-307.

\bibitem{Kot}
B. Kotschwar,  {\it On rotationally invariant shrinking gradient Ricci
solitons,} Pacific Journal of Mathematics, 2008, \textbf{236} (1): 73-88.

\bibitem{Le}
P.-F. Leung, {\it  On the consecutive eigenvalues of the Laplacian of a compact minimal submanifold in
a sphere,} Journal of the Australian Mathematical Society (Series A),  1991, {\bf 50}(03): 409-426.

\bibitem{L}
P. Li, {\it  Eigenvalue estimates on homogeneous manifolds,} Commentarii Mathematici Helvetici, 1980, {\bf 55}(1): 347-363.

\bibitem{LY}
P. Li, S.T. Yau. {\it  Estimates of eigenvalues of a compact Riemannian manifold,} Geometry of the
Laplace operator (Proc. Sympos. Pure Math., Univ. Hawaii, Honolulu, Hawaii, 1979), Proc.
Sympos. Pure Math., XXXVI, American  Mathematical Society, Providence, R.I., 1980, pp. 205-239.

\bibitem{LW}
H. Li and Y. Wei, {\it  $f$-minimal surface and manifold with positive
$m$-Bakry-\'{E}mery Ricci curvature,} The Journal of Geometric Analysis, 2015, \textbf{25}(1): 421-435

\bibitem{Lic1}
A. Lichnerowicz, Vari\'{e}t\'{e}s riemanniennes \`{a} tenseur C non
n\'{e}gatif, C. R. Acad. Sci. Paris, 1970, {\bf 271}: 650-653.

\bibitem{Lic2}
A. Lichnerowicz,  Vari\'{e}t\'{e}s k\"{a}hl\'{e}riennes \`{a}
premi\`{e}re classe de Chern non negative et vari\'{e}t\'{e}s
riemanniennes \`{a} courbure de Ricci g\`{e}n\`{e}ralis\`{e}e non
negative,  Journal of Differential Geometry, 1971, {\bf 6}(1): 47-94,.

\bibitem{Ling}
J. Ling. {\it  Lower bounds of the eigenvalues of compact manifolds with positive Ricci curvature.}
Annals of Global Analysis and Geometry, 2007, {\bf31}, (4): 385-408.

\bibitem{MC}
L. Ma and D. Chen, {\it Remarks on complete non-compact
gradient expanding Ricci solitons,} Kodai Mathematical Journal, 2010,
{\bf 33}(2): 173-181.

\bibitem{MD}
L. Ma and S.-H. Du, {\it  Extension of Reilly formula with applications to
eigenvalue estimates for drifting Laplacins,} Comptes Rendus Mathematique, 2010,
\textbf{348} (21-22): 1203-1206

\bibitem{ME1}
F.-L. Manuel  and G.-R. Eduardo, {\it  A remark on compact Ricci solitons,} Mathematische Annalen, 2008, {\bf 340} (4): 893-896.

\bibitem{ME2}
F.-L. Manuel  and G.-R. Eduardo, {\it Rigidity of shrinking Ricci solitons,} Mathematische Zeitschrift, 2011,  {\bf 269}:461-466.


\bibitem{MS}
O. Munteanu, N. Sesum, {\it On gradient Ricci solitons},  Journal of Geometric Analysis, 2013, {\bf 23}(2):   539-561.
\bibitem{Melv}
M. S. Berger, {\it  Nonlinearity and Functional Analysis,} Academic Press, New York-San Francisco-London, 1977.

\bibitem{ME}
F.-L. Manuel and G.-R. Eduardo, {\it  Rigidity of shrinking Ricci solitons,} Mathematische Zeitschrift, 2011, {\bf 269}(1-2): 461-466.

\bibitem{MW1}
O. Munteanu and J. Wang, {\it  Smooth metric measure spaces with
nonnegative curvature,} Communications in Analysis and Geometry, 2011, \textbf{19} (3), 451-486

\bibitem{MW2}
O. Munteanu and J. Wang, {\it  Analysis of weighted Laplacian and
applications to Ricci solitons}, Communications in Analysis and Geometry, 2012, \textbf{20} (1): 55-94

\bibitem{MS}
O. Munteanu and N. Sesum, {\it  On Gradient Ricci Solitons,}  Journal of Geometric Analysis, 2013, {\bf 23} (2): 539-561.

\bibitem{N}
A. Naber, {\it  Noncompact shrinking four solitons with nonnegative curvature,}Journal f\"{u}r die reine und angewandte Mathematik (Crelles Journal), 2010, {\bf 2010}(645): 125-153.

\bibitem{PPW1}
L. E. Payne, G. Polya and H. F. Weinberger, {\it  Sur le quotient de deux fr\'{e}quences propres cons\'{e}cutives},
Comptes Rendus Hebdomadaires Des Seances De L Academie Sciences, 1955, {\bf 241}: 917-919.

\bibitem{Pere1}
 G. Perelman, {\it  The entropy formula for the Ricci flow and its geometric applications,} arXiv: math/0211159,
2002.

\bibitem{PPW2}
L. E. Payne, G. Polya and H. F. Weinberger, {\it On the ratio of consecutive eigenvalues},  Journal of Mathematics and Physics,  1956, {\bf 35} (1): 289-298.

\bibitem{PW1}
P. Petersen and W. Wylie, {\it Rigidity of gradient Ricci solitons,}  Pacific journal of mathematics,  2009, {\bf 241}(2): 329-345.

\bibitem{PW2}
P. Petersen and W. Wylie, {\it  On gradient Ricci solitons with symmetry,} Proceedings of the American Mathematical Society, 2009, {\bf 137}(6): 2085-2092.

\bibitem{PRS}
S. Pigola, M. Rimoldi and A. G. Setti, {\it  Remarks on non-compact
gradient Ricci solitons,}  Mathematische Zeitschrift, 2011, \textbf{268}(3-4): 777-790.

\bibitem{SY}
R. Schoen, S.-T. Yau, {\it Lectures on Differential Geometry},
International Press, 1994.


\bibitem{S}
 A. G. Setti, {\it  Eigenvalue estimates for the weighted Laplacian on a Riemannian manifold,}
Rendiconti del Seminario Matematico della Universit\`{a} di Padova, 1998, {\bf100}: 27-55.

\bibitem{SHI}
A. El Soufi, E.M. Harrell II, S. Ilias. {\it Universal inequalities for the eigenvalues of Laplace and
Schr\"{o}dinger operators on submanifolds,} Transactions of the American Mathematical Society,  2009, {\bf 361}(5): 2337-2350.

\bibitem{SI}
A. El Soufi and S. Ilias, {\it  Second eigenvalue of Schr\"{o}dinger operators and mean curvature,} Communications in Mathematical Physics, 2000, {\bf 208} (3): 761-770.

\bibitem{SZ}
Y. Su and H. Zhang, {\it Rigidity of manifolds with Bakry-\'{E}mery Ricci
curvature bounded below,} Geometriae Dedicata, 2012, \textbf{160} (1): 321-331


\bibitem{W}
L.-F. Wang, {\it   The upper bound of the $L^{2}_{\mu}$ spectrum,} Annals of Global Analysis and Geometry, 2010, \textbf{37}(4): 39-402

\bibitem{WX}
Q. Wang and C. Xia,  {\it  Universal bounds for eigenvalues of Schr\"{o}dinger operator on
Riemannian manifolds,} Annales Academi{\ae} Scientiarum Fennic{\ae}
Mathematica , 2008, {\bf 33}: 319-336.

\bibitem{Wey}
H. Weyl, {\it Der Asymptotische Verteilungsgesetz der Eigenwerte Linearer partieller Differentialgleichungen,} Mathematische Annalen, 1912, {\bf 71}(4): 441-469.

\bibitem{WW}
G. Wei and W. Wylie, {\it   Comparison geometry for the Bakry-\'{E}mery
Ricci tensor,} Journal of Differential Geometry, 2009, \textbf{83} (2): 377-405

\bibitem{W1}
J.-Y. Wu, {\it Upper bounds on the first eigenvalue for a diffusion
operator via Bakry-\'{E}mery Ricci curvature,}  Journal of Mathematical Analysis and Applications, 2010,
\textbf{361} (1): 10-18

\bibitem{W2}
J.-Y. Wu, {\it Upper bounds on the first eigenvalue for a diffusion
operator via Bakry-\'{E}mery Ricci curvature II,} Results in Mathematics, 2013, \textbf{63} (3-4): 1079-1094.


\bibitem{XX}
C. Xia and H. Xu, {\it Inequalities for eigenvalues of the drifting Laplacian on Riemannian manifolds}, Annals of Global Analysis and Geometry, 2014, {\bf 45} (3): 155-166.

\bibitem{Y}
H. C. Yang, {\it An estimate of the difference between consecutive eigenvalues, preprint IC/91/60 of
 International Centre for Theoretical Physics, Trieste,} 1991.


\bibitem{YY}
P. C. Yang and S. T. Yau, {\it Eigenvalues of the Laplacian of compact Riemannian surfaces and minimal
submanifolds},  Annali della Scuola Normale Superiore di Pisa-Classe di Scienze, 1980, {\bf 7}: 55-63.

\bibitem{Z}
L. Zeng, {\it The eigenvalue problems on Riemannian manifolds}, Doctoral dissertation, (2013), Saga University, http://portal.dl.saga-u.ac.jp/handle/123456789/121307.

\bibitem{Z1}
L. Zeng, {\it Eigenvalues of the drifting Laplacian on complete noncompact Riemannian manifolds,} Nonlinear Analysis: Theory, Methods \& Applications, 2016,
{\bf 141}: 1-15.
\bibitem{Z2}
L. Zeng, {\it  Estimates for the Eigenvalues of the Drifting Laplacian on Some Complete Ricci Solitons}, preprint.

\bibitem{Z3}
L. Zeng, {\it The gaps between consecutive eigenvalues of Laplacian on Riemannian manifolds}, preprint.





\end{thebibliography}
\end{document}